\documentclass[10pt]{amsart}
\usepackage{pb-diagram,amssymb,epic,eepic,verbatim}
\usepackage{color,graphicx}
\usepackage{epsfig,psfrag}
\usepackage{amsmath}
\usepackage{geometry}

\DeclareMathAlphabet{\mathpzc}{OT1}{pzc}{m}{it}
\newcommand\pz{\mathpzc}

\newtheorem{theorem}{Theorem}[subsection]

\newtheorem{proposition}[theorem]{Proposition}
\newtheorem{lemma}[theorem]{Lemma}
\newtheorem{lem}[theorem]{}
\theoremstyle{definition}
\newtheorem{definition}[theorem]{Definition}
\theoremstyle{remark}
\newtheorem{remark}[theorem]{Remark}
\newtheorem{example}[theorem]{Example}
\newcommand{\blem}{\begin{lem} \rm}
\newcommand{\elem}{\end{lem}}
\newtheorem{altdef}[theorem]{\textbf{Alternative Definition}}
\newcommand{\C}{\mathbb{C}}

\renewcommand{\P}{\mathbb{P}}

\newcommand{\on}{\operatorname}

\renewcommand{\ker}{ \on{ker}}

\newcommand\dirac{/\kern-1.2ex\partial} 
\newcommand\qu{/\kern-.7ex/} 
\newcommand\lqu{\backslash \kern-.7ex \backslash} 

\newcommand\dr{r_+ \kern-.7ex - \kern-.7ex r_-}

\newcommand{\bq}{\begin{equation}}
\newcommand{\eq}{\end{equation}}

\newcommand{\ra}{\rightarrow}

\newcommand\sig{\sigma}

\newcommand\bdefn{\begin{definition}}
\newcommand\edefn{\end{definition}}
\newcommand\bea{\begin{eqnarray*}}
\newcommand\eea{\end{eqnarray*}}
\newcommand\bcv{\left[ \begin{array}{r} }
\newcommand\ecv{\end{array} \right] }

\newcommand\bma{\left[ \begin{array} }
\newcommand\ema{\end{array} \right]}
\newcommand\ben{\begin{enumerate}}
\newcommand\een{\end{enumerate}}
\newcommand\beq{\begin{equation}}
\newcommand\eeq{\end{equation}}
\newcommand\bex{\begin{example}}
\newcommand\bsj{\left\{ \begin{array}{rrr} }
\newcommand\esj{\end{array} \right\}}

\newcommand\eex{\end{example}}

\newcommand\sx{*\kern-.5ex_X}

\newcommand\Fuk{{\on{Fuk}}}

\renewcommand\P{\pz{C}}
\newcommand\khs{\pz{Kh}_{symp}}
\newcommand\Sm{\mathcal{S}_m}
\newcommand\nbhd  {neighborhood }
\newcommand\cal {\mathcal}
\newcommand\bb {\mathbb}

\begin{document}

\title{Seidel-Smith cohomology for Tangles}

\author{ Reza Rezazadegan }

\begin{abstract}
We generalize the``symplectic Khovanov cohomology'' of Seidel and Smith \cite{SS} to tangles using the notion of symplectic valued topological field theory introduced by Wehrheim and Woodward \cite{WW}.
\end{abstract}

\maketitle

\section{Introduction}\label{intro}

In the year 2000, Mikhail Khovanov \cite{categorification} introduced a doubly graded homology theory of links which categorifies the Jones polynomial, i.e. whose graded Euler characteristic
equals the Jones polynomial. He later generalized this cohomology theory to even tangles (i.e. tangles with even number of initial and endpoints)\cite{functorvalued}. In that paper he introduced a family of rings $\{H^n\}_{n\in\mathbb{N}}$ and to each $(m,n)$-tangle he assigned  a complex of $(H^m,H^n)$ bimodules. By means of tensor product one obtains, for each tangle, a functor on the category of such complexes modulo  homotopy.
In 2003, P.~Seidel and I.~Smith defined a singly graded link cohomology based on symplectic geometry  which they called ``symplectic Khovanov homology''.  They conjectured that this invariant equaled  Khovanov homology after the collapse of the bigrading. They defined a family $\{\cal{Y}_n\}_{n\in\mathbb{N}}$ of symplectic manifolds and to each braid $\beta\in Br_{2m}$ they assigned a symplectomorphism  $h_\beta$ of $\cal{Y}_m$. (See  section \ref{reviewss}.) The Seidel-Smith invariant of a link $K$
is 
the Floer cohomology of $L$ and $h_\beta(L)$ where $L$ is a specific Lagrangian submanifold of $\cal{Y}_m$ and $\beta$ is any braid representation of $K$. They prove that this is independent of the choice of the braid representation $\beta$.
Later C.~Manolescu \cite{ChiprianSS} gave a more explicit description of the invariant and equipped the chain complex with a second grading, showing that the Euler characteristic of this chain complex equals the Jones polynomial. However it is not known if this grading descends to a grading on cohomology.

In this paper we construct a generalization of the Seidel-Smith invariant to even tangles. 
 To any elementary $(i,j)$-tangle ${T}$ we assign a Lagrangian correspondence $L_T$ between $\cal{Y}_i$ and $\cal{Y}_j$. If $T$ is a braid, we assign to it the graph of the symplectomorphism $h_T$ defined by Seidel and Smith. The remaining  elementary tangles are caps and cups. (See Figure \ref{elementangles}.) To a $(m,m+2)$ cap we assign a vanishing cycle over the diagonal $\Delta\subset \cal{Y}_m \times \cal{Y}_m$. The Lagrangian assigned to a cup is the transpose of this vanishing cycle. See section \ref{func}.
Now any given $(m,n)$-tangle $\cal{T}$ can be decomposed into a composition of elementary ones $$\cal{T}=T_k T_{k-1}\cdots T_1.$$  To $\cal{T}$ we assign the \textit{generalized} Lagrangian correspondence  $$\Phi(\cal{T})=(L_{T_k},L_{T_{k-1}}, \cdots, L_{T_1})$$ between  $\cal{Y}_m$ and $\cal{Y}_n$.
We then prove the following.\\
\\
\large{\textbf{Theorem}} \ref{wellfunc}. \textit{Up to isomorphism of generalized correspondences, $\Phi(\cal{T})$ is independent of the decomposition of $\cal{T}$ into elementary tangles.}\\

This way we obtain two invariants for each $(m,n)$-tangle $T$; The first one is a functor $\Phi_T^\#$ from the generalized Fukaya category of $\cal{Y}_m$ to that of $\cal{Y}_n$. The category used here is an enlargement of the Fukaya category of a Stein manifold to include a special class of noncompact Lagrangians.
The second one is a graded abelian group, denoted $\khs(T)$, which is, roughly, the Floer cohomology of $\Phi(T)$. For this second invariant to be well-defined we first have to deal with the compactness of the involved moduli spaces. The reason is that the Lagrangians assigned to caps and cups are not compact. We prove compactness using standard (but not very well-known) arguments on Lagrangians in manifolds with contact type boundary. In sections \ref{laginstein} and \ref{floercoho} we put together necessary tools for construction of Floer homology of noncompact Lagrangians in Stein manifolds.
From these plus Theorem \ref{wellfunc} and  the Functoriality Theorem  of \cite{WW} we get the following.\\
\\
\large{\textbf{Theorem}} \ref{khswelldef}. \textit{$\khs(T)$ is well-defined and is independent of the decomposition of $T$ into elementary tangles.}\\

An algebraic-geometric equivalent of Khovanov homology has been developed by S.~Cautis and J.~Kamnitzer \cite{CautisK1}. In that paper they assign  to each elementary $(k,l)$-tangle a  Fourier-Mukai kernel between specific algebraic varieties $Y_k$ and $Y_l$. This way they assign to each $(m,n)$-tangle $T$ a functor $\Psi(T)$ from the bounded derived category of equivariant  coherent sheafs on $Y_m$ to that of $Y_n$. They prove that for a link $K$ the cohomology of the chain complex $\Psi(K)(\mathbb{C})$ equals the Khovanov homology of $T$ with diagonal grading.
The functors $\Psi(T)$ and our $\Phi_T^\#$ are expected to be related by mirror symmetry. 
Joel Kamnitzer \cite{Kamnitzer} has recently proposed a method for categorifying all link polynomials from quantum groups.
In this
picture, for a complex reductive group $G$, the symplectic
fibration used by Seidel and Smith ( which is in fact the adjoint quotient map) is replaced by a 
fibration whose total space is the Beilinson-Drinfeld Grassmannian. This
Grassmannian is, roughly speaking, the moduli space of $G^{\vee^{}}$-bundles
on $\mathbb{P}^1$ which are trivial on the complement of a finite set of
points. Here $G^{\vee^{}}$ is the Langlands dual of $G$. When two such points approach each other, one has a similar situation
to that of Seidel-Smith where two eigenvalues come together. Kamnitzer proves
a local neighborhood theorem analogous to that of Seidel and Smith. \\

In a forthcoming paper we show that $\khs(T)$ for an $(m,n)$-tangle is a  bimodule over  $(H^m,H^n)$ and that it  is 
equivalent to Khovanov homology   for $T$ flat (ie. crossingless). We believe that extending the symplectic link invariant to tangles makes the comparison between the combinatorial (Khovanov) and the symplectic (Seidel-Smith) invariants easier and so serves as a small step in understanding the capability of symplectic geometry in  the categorification paradigm.\\



\textbf{Acknowledgments.} I am grateful to my adviser, Christopher Woodward, who suggested the use of generalized Lagrangian correspondences to generalize the Seidel-Smith invariant and who was essential to the formation of this paper. I would also like to thank Eduardo Gonzalez, Ciprian Manolescu, Paul Seidel,  Charles Weibel and the referee of this paper. 


\section{The Work of Seidel and Smith}\label{reviewss}

  In this section we review the construction of Seidel and Smith which we will make use of in the rest of this paper. Most proofs are omitted. The reader is referred to \cite{SS} for  details.  
Denote by $\mathrm{Conf}_m$ the space of all unordered $m$-tuples of  distinct complex numbers  $(z_1,\cdots,z_m).$ Denote by $\mathrm{Conf}^0_m$ the subset of  $\mathrm{Conf}_m$ consisting of $m$-tuples which add up to zero, i.e. $z_1+\cdots+z_m=0$.

\subsection{Transverse slices} The basic reference for this section is \cite{Slodowy}.
Let $G$ be a complex semisimple Lie group and consider the adjoint action of $G$ on its Lie algebra $\mathfrak{g}$. The adjoint quotient map  \smash{$\chi:\frak{g}\rightarrow\frak{g}/G$} sends each element of $\frak{g}$ to its orbit in \smash{$\frak{g}/G$}. A theorem of Chevalley (See \cite{humphreysbook}, Chapter 23) asserts that \smash{$\frak{g}/G$} can be identified with \smash{$\frak{h}/W$} where $\frak{h}$ is a Cartan subalgebra of $\frak{g}$ and $W$ the associated Weyl group. Therefore $\chi$ can be regarded as assigning to each $y\in \frak{g}$ the eigenvalues (or equivalently coefficients of the characteristic polynomial) of the semisimple part of $y$.
\begin{definition}
 A \textit{transverse slice} for the adjoint action  at $x\in \frak{g}$ is a local complex submanifold $\cal{S}$ of $\frak{g}$ containing $x$ which is transverse to the orbit of $x$ and is invariant under the action of the isotropy subgroup $G_x$.
\end{definition}
It is obvious that such an $\cal{S}$ intersects the orbit of any $y$ sufficiently close to $x$ transversely. If $K$ is a local submanifold of $G$  containing the identity such that $T_e K$ is complementary to $\{y\in \frak{g}: [x,y]=0\}$ then it can be easily seen that any other transverse slice
at $x$ lies (locally) in the image of  the map \begin{equation}\label{AdK}\operatorname{Ad}: K\times S\rightarrow \frak{g}.\end{equation}
The Jacobson-Morozov lemma \cite{JMorozov} tells us that if $x\in \frak{g}$ is nilpotent then there are elements $y,h\in \frak{g} $ such that
$$[x,y]=h  \qquad [x,h]=2x  \qquad [y,h]=-2y.$$

Consider the vector field $K$ on $\mathfrak{g}$ given by $\label{fieldK}K(z)=2z-[h,z].$  It defines a $\mathbb{C}^*$ action on $\frak{g}$ given by $\lambda_r(z)=r^2 e^{-\log(r)}ze^{\log(r)h}$ for $r\in \mathbb{C}^*$. The vector field $K$ vanishes at $x$ so $x$ is a fixed point of $\lambda_r$.
A slice at $x$ is called \textit{homogeneous} if it is invariant under $\lambda_r$.

Now we specialize to $\frak{g}=\frak{sl}_{2m}=\frak{sl}_{2m}(\mathbb{C})$. In this case $W=S_n$.  
   Take $x$ to be a nilpotent Jordan block of size $2m$.
Let $\mathcal{S}_m$ be the set of matrices in $\mathfrak{sl}_{2m}$ of the form  \\ \bq\label{matrixS}\left(
                                         \begin{array}{cccccc}
                                           y_1 & 1 &\space & & \space& \\
                                           y_2 & & 1 &  & & \\
                                           \vdots&  &  & \ddots & \\
                                             y_{n-1}& & & & 1\\
                                             y_n &  &  & & 0\\
                                         \end{array}
                                       \right)\eq\\
where $1$ is the $2\times 2$ identity matrix, $y_1\in \mathfrak{sl}_2$ and $y_i\in \mathfrak{gl}_2$ for $i>1$. It is easy to see that $\Sm$ is a homogeneous slice to the orbit of $x$.
  $\chi$ restricted to $\mathrm{Conf}^0_{2m}$ 
  is a differentiable fiber bundle \cite{Slodowy}.
   We denote the fiber of $\chi$ over $t$ by $\mathcal{Y}_{m,t}$, i.e. $\cal{Y}_{m,t}=\chi^{-1}(t)$. If $t=(\mu_1,\ldots,\mu_{2m}) \notin \mathrm{Conf}^0$, by  $\cal{Y}_{m,t}$ we mean $\cal{Y}_{m,t'}$  where $t'=(\mu_1-\sum \mu_i/2m,\ldots, \mu_{2m}-\sum \mu_i/2m).$
Let $E_{y}^\mu$ denote the $\mu$-eigenspace of $y$.

\begin{lemma}\label{2dkernel}
  For any $y\in \Sm$  and $\mu\in \mathbb{C}$ the projection $\mathbb{C}^{2m}\rightarrow \mathbb{C}^2$ onto the first two coordinates gives an injective map $E_y^\mu\rightarrow \mathbb{C}^2$.
  Any eigenspace of any element $y\in \Sm$ has dimension at most two. Moreover the set of elements of $\mathcal{S}_m$ with 2 dimensional kernel can be canonically identified with $\mathcal{S}_{m-1}$ and this identification is compatible with $\chi$.
\end{lemma}
\begin{proof}
If not then  the intersection of $E_y^\mu$ with $\{0\}^2\times\mathbb{C}^{2m-2} $ is nonzero. Applying the $\mathbb{C}^*$ action we see that the same holds for \smash{$E_{\lambda_r(y)}^{r^2\mu}$}. As $r$ goes to zero, $\lambda_r(y)\rightarrow x$ so we get $\dim \ker  x>2$ which is contradiction.
From this injectivity we see that each element of $\ker y$ is determined by its first two coordinates so if $\dim  \ker y=2$ then $y_m=0$ and vice versa.
The subset of such matrices is identified with $\cal{S}_{m-1}$.
\end{proof}

For any subset $A\subset \mathfrak{sl}_{2m} $, let  $A^{sub,\lambda}$ (resp. $A^{sub3,\lambda}$) be the subset of matrices in $A$ having eigenvalue $\lambda$  of multiplicity two (resp. three) and two Jordan block of size one (resp. two Jordan blocks of sizes 1,2) corresponding to the eigenvalue $\lambda$  and no other coincidences between the eigenvalues. Here are two results describing  neighborhoods of $\mathcal{S}_m^{sub,\lambda} $ and  $\mathcal{S}_m^{sub3,\lambda} $ in $\Sm$.

\begin{lemma}\label{2nbhd}
Let $D\subset \on{\mathrm{Conf}}^0_{2m}$ be a disc consisting of the $2m$-tuples $(\mu-\epsilon, \mu-\epsilon,\mu_3,\ldots, \mu_{2m})$ with $\varepsilon$ small. Then there is a neighborhood $U_\mu$ of $\mathcal{S}_m^{sub,\mu} $ in $\Sm\cap \chi^{-1}(D)$ and  an isomorphism $\phi$ of $U_\mu$  with  a \nbhd of $\Sm^{sub,\mu}$ in $\Sm^{sub,\mu}\times \mathbb{C}^3$ such that $f \circ \phi=\chi$ on $\Sm\cap \chi^{-1}(D)$ where $f(x,a,b,c)=a^2+b^2+c^2$.
  Also if $N_y \Sm^{sub,\mu}$ denotes the normal bundle to $\Sm^{sub,\mu}$ at $y$ the we have \bq\label{2nbhdmap}\phi(N_y \Sm^{sub,\mu})= \mathfrak{sl}(E_{y}^{\mu})\oplus \zeta_y\eq where 
 $\zeta_y$ is the trace free part of $\{\mathbb{C}\subset \mathfrak{gl}(E_y^\mu) \} \oplus \mathfrak{gl}(E_y^{\mu_3})\oplus \ldots \oplus \mathfrak{gl}(E_y^{\mu_{2m}}) $.   
\end{lemma}

\begin{proof}
For $y\in \Sm^{sub,\mu}$, let $\cal{S}_y$ be a  subspace of $T_y \Sm$ complementary  to $T_y \Sm^{sub,\lambda} $ which depends holomorphically on $y$. These subspaces together form a tubular \nbhd of $\Sm^{sub,\mu}$ in $\Sm$. Since $\Sm$ and $\frak{sl}^{sub,\mu}_{2m}$ intersect transversely, $\cal{S}_y$ is also a transverse slice at $y$ for the adjoint action on $\frak{sl}_{2m}$. We can produce another family of transverse slices $\cal{S}'_y$  by setting $\cal{S}'_y=\frak{sl}(E_{y}^{\mu})\oplus \zeta_y$ which equals the trace free part of $\frak{gl}(E_y^\mu)\oplus  \mathfrak{gl}(E_y^{\mu_3})\oplus \cdots \oplus \mathfrak{gl}(E_y^{\mu_{2m}}) $.
 The reason is that $[y,\frak{sl}_{2m}]=0\oplus \frak{sl}_{2m-2}^0$
where the first component consists of zero in $\frak{sl}_2$ and the second one consists of matrices with zeros on the diagonal and the right hand side is transverse to $\cal{S}'_y$.

Now  $\cal{S}_y$ is isomorphic (as local complex manifolds) to $\cal{S}'_y$ for each $y$ with an isomorphism that moves points only inside their adjoint orbits (and hence is compatible with $\chi$). We can choose these isomorphisms to depend holomorphically on $y$.
Each $\mathfrak{gl}(E_y^{\mu_i})\subset \frak{sl}_{2m}$ for $i>2$ can be canonically (without choice of a basis) identified with $\mathbb{C}$. Lemma \ref{2dkernel} tells us that $E_y^\mu$ can be canonically identified with $\mathbb{C}^2$ so $\mathfrak{sl}(E_y^\mu)$ is identified with $\frak{sl}_2$. It follows that $\cal{S}'_y\cong \frak{sl}_2 \oplus \mathbb{C}^{2m-2}.$ The desired $\phi$ is the composition of the two isomorphisms in this paragraph.
\end{proof}


\begin{remark}
If $y$ has two linearly independent $\mu_1$ eigenvectors as well as two linearly independent $\mu_2$ eigenvectors and with no other coincidences between the eigenvalues, we can repeat the above argument to obtain
\bq\label{22nbhd}\phi(N_y (\Sm^{sub,\mu_1}\cap \Sm^{sub,\mu_2}))= \mathfrak{sl}(E_{\mu_1})\oplus \mathfrak{sl}(E_{\mu_2})\oplus \zeta.\eq
So $\phi$  gives an isomorphism between a \nbhd of $\Sm^{sub,\mu_1}\cap \Sm^{sub,\mu_2}$  in $\Sm$ and $(\Sm^{sub,\mu_1}\cap \Sm^{sub,\mu_2})\times\mathbb{C}^3\times \mathbb{C}^3$.
\end{remark}

  Consider the line bundle $\mathcal{F}$ on $\Sm^{sub3,\mu}$ whose fiber at $y\in\Sm^{sub3,\mu}$ is $(y-m)E_{y_s}^\mu$ where $y_s$ is the semisimple part of $y$. To $\cal{F}$ one  associates a  $\Bbb{C}^4$ bundle $\mathcal{L}=(\mathcal{F}\backslash 0)\times_{\mathbb{C}^*} \mathbb{C}^4$ where  $z\in \mathbb{C}^*$ acts on $\mathbb{C}^4$  by
  \begin{equation}\label{C2action}
  (a,b,c,d) \rightarrow (a,z^{-2}b, z^{2}c, d ).
\end{equation}
$\cal{L}$ decomposes as \bq\label{Ldecomp}\cal{L}\cong\mathbb{C}\oplus\cal{F}^{-2}\oplus\cal{F}^2\oplus \mathbb{C}.\eq
Fibers of $\cal{L}$ should be regarded as transverse slices in $\frak{sl}_3$. Upon choosing  suitable coordinates on such a transverse  slice (at the zero matrix) the function $\chi$  equals the function $p:\frak{sl}_3\rightarrow \mathbb{C}^2$ given by \bq\label{p} p(a,b,c,d)=(d,a^3-ad+bc).\eq  $p$ is also well-defined as a function $\cal{L}\rightarrow \mathbb{C}^2$ because $b$ and $c$ are coordinates on line bundles which are inverses of each other.  
Denote by $\tau(d,z)$ the set of solutions of  $\lambda^3-d\lambda+z=0$.

\begin{lemma}\label{3nbhd}
 Let $P\subset \mathrm{Conf}^0_{2m}$ be the set of $2m+2$-tuples \bq\label{3nbhddisk}(\mu_1,\ldots, \mu_{i-1},\tau(d,z),\mu_{i+3},\ldots ,\mu_{2m+2}).\eq where $d$ and $z$ vary in a small disc in $\mathbb{C}$ containing the origin.
 There is a \nbhd  $V$ of $\Sm^{sub3}$ in $\Sm\cap\chi^{-1}(P)$ and an isomorphism $\phi'$ from $V$ to a \nbhd of zero section in $\mathcal{L}$ such that  $p(\phi'(x))=(d,z)$ if \\
 $$\chi(x)=(\mu_1,\ldots, \mu_{i-1},\tau(d,z),\mu_{i+3}, \ldots, \mu_{2m+2}).$$
\end{lemma}

\subsection{Relative vanishing cycles }\label{relvanish}

Let $X$ be a complex manifold and $K$ a compact submanifold. Let  $g$ be a Kahler metric on $Y=X \times \mathbb{C}^3$ (not necessarily the product metric)  and denote its imaginary part by $\Omega$. Consider the map $f:X\times \mathbb{C}^3\rightarrow \mathbb{C}$ given by $f(x,a,b,c)= a^2+b^2+c^2$ and denote by $\phi_t$ the gradient flow of $-\on{Re} f$. 
Let $W$ be the set of points  $y\in Y$ for which the trajectory $\phi_t(y)$ 
 exists for all positive $t$. 

\begin{lemma}
$W$ is a manifold and $l$ is smooth.  The function $l:W\rightarrow X$ given by $l(y)=\lim_{t\rightarrow \infty} \phi_t(y)$ is well-defined and smooth.   We have $\Omega|_W=l^* \Omega|_X$. $f$ restricted to $W$ is real and nonnegative.
\end {lemma}
\begin{proof}The first two assertions follow from stable manifold theorem (Theorem 1 in \cite{Stablemanifoldthm}) and the rest from the fact that gradient vector field of $-\on{Re} f$ is the Hamiltonian vector field of $\on{Im} f$.
\end{proof}

Set $V_t(K)=\pi^{-1}(t)\cap l^{-1}(K)=l|_{\pi^{-1}(t)}^{-1}(K)$ which is a manifold for $t$ small.  
It follows from Morse-Bott lemma that $V_t(K)$ is a 2-sphere bundle on $K$ for $t$ small. To generalize the invariant to tangles we will need a slightly more general version of the above construction in which $K$ is noncompact and the metric equals the product metric outside a compact subset. (See section \ref{func}.) The resulting vanishing cycle equals (symplectically) the product bundle outside a compact subset.

\subsection{Fibered $A_2$  singularities}\label{fibered}

Assume we have the same situation as in the Lemma \ref{3nbhd}, i.e. let $\cal{F}$ be a holomorphic line bundle over a complex manifold $X$ and define $Y$ to be $(\cal{F}\backslash0) \times_{\mathbb{C}^*} \mathbb{C}^4$ where the $\mathbb{C}^*$ action is as in the formula \eqref{C2action}. Let $\Omega$ be an arbitrary Kahler form on $Y$ and by  regarding $X$ as the zero section of $Y$, $\Omega$ restricts to a Kahler form on $X$. Let $(a,b,c,d)$ be the coordinates on fibers of $Y\rightarrow X$ and $(d,z)$ coordinates on $\mathbb{C}^2$. Let the map $p:Y\rightarrow \mathbb{C}^2$ be as in Lemma \ref{3nbhd}. Let $Y_d=p^{-1}(\mathbb{C}\times \{d\})$ and $p_d: Y_d\rightarrow \mathbb{C}$ be the restriction of $p$. Set $Y_{d,z}=p^{-1}(d,z)$. For $d\neq 0$ critical values of $p_d$ are \smash{$\zeta^{\pm}_d=\pm 2\sqrt{d^3/27}$}.

 Let $K$ be Lagrangian submanifold of $X$.  Using relative vanishing cycle construction for the function $p_d$ we can obtain a Lagrangian submanifold $L_d$ of $Y$ which is a  sphere bundles over $K$. (This construction works when $Y$ is a nontrivial bundle over $X$ as well.) There is another way of describing this  Lagrangian as follows. Let $\mathrm{Y}\cong \mathbb{C}^4$  be the fiber of $Y\rightarrow X$ over some point of $X$ and let $p:\mathrm{Y}\rightarrow \mathbb{C}^2$ be as before. The restriction of the $\mathbb{C}^*$ action to $S^1$ is a Hamiltonian action with the moment map $\mu(a,b,c,d)=|c|^2-|b|^2$.
 Define
 \bq\label{C} C_{d,z,a}=\{(b,c):  \mu(b,c)=0, a^3-da-z=-bc \}\subset \mathrm{Y}_{d,z} \eq which is a point if $a^3-da-z= 0$ and a circle otherwise.  The three solutions of this equation correspond to the critical values of the projection $q_{d,z}:\mathrm{Y}_{d,z}\rightarrow \mathbb{C}$ to the $a$ plane. In the situation of Lemma \ref{3nbhd} they correspond to the three  eigenvalues of $\mathrm{Y}$. Let $\alpha(r)$ be any embedded curve in $\mathbb{C}$ which intersects these critical values (only) if $r=0,1$. Define \bq\label{Lambd}\Lambda_\alpha=\bigcup_{r=0}^{1}  C_{d,z,\alpha(r)}\eq which is a Lagrangian submanifold of $\mathrm{Y}_{d,z}$ (with Kahler form induced from $\mathbb{C}^4$). Let $\mathbf{c}, \mathbf{c}',\mathbf{c}''$ be as in the Figure \ref{curvesalphabeta} where dots represent the critical values of $q_{d,z}$.
We can associate to $K$ a Lagrangian submanifold $\Lambda_{d,\alpha}$ of $Y$ by defining $\Lambda_{d,\alpha}=(Y|K)\times_{S^1} \Lambda_{\mathbf{c}}$.  
  Seidel and Smith prove that these two procedures give the same result:

 \begin{lemma}\label{2lagrang}
 If the Kahler form on $Y$ is obtained from  a Kahler form on $X$, a Hermitian metric on $\mathcal{F}$ and the standard form on $\mathbb{C}^4$ then $L_d=\Lambda_{d,\mathbf{c}}$.

 \end{lemma}

\begin{figure}[ht]
\begin{center}
\scalebox{.7}{
\input{curvesalphabeta.pstex_t}
}
\caption{}
 \label{curvesalphabeta}
\end{center}
\end{figure}

\subsection{Symplectic structure}\label{symplstr}
The symplectic structure that Seidel and Smith use on  $\Sm$ is not the standard structure on $\Sm\cong \mathbb{C}^{4m-1}$. The reason is to obtain well-defined parallel transport maps for the fibration $\chi$ whose fibers are noncompact. We need a plurisubharmonic function whose fiberwise critical point set project properly under $\chi$.
Fix $\alpha>m$.
Let $\xi_i(z)=|z|^{2\alpha/i}$ for $z\in \mathbb{C}$. These functions are $C^1$ however by adding  compactly supported functions $\eta_i$ we can obtain $C^\infty$ functions $\psi_i=\xi_i+\eta_i$ on $\mathbb{C}$. We choose $\eta_i$ such that $-dd^c\psi_i>0$.   Let $\psi$ be the function on $\Sm$ whose value at  $y\in \mathcal{S}_m$ is  $$\sum_i \sum_{\mu,\nu\in\{0,1\}} \psi_i((y_{1i})_{\mu,\nu}).$$
We can choose $\eta_i$ so that $\psi$ is an exhausting plurisubharmonic function on $\mathcal{S}_m$ which gives us the symplectic form $\Omega=-dd^c\psi$. Outside a set, which is the product of the complement of a compact set in each coordinate plane, we have $$\Omega=4\sum_i (\alpha/i)(\alpha/i-1) \sum_{\mu,\nu\in\{0,1\}} |(y_{1i})_{\mu,\nu}|^{2\alpha/i} d(y_{1i})_{\mu,\nu}\wedge d\overline{ (y_{1i})_{\mu,\nu}}.$$

 By restriction we obtain Stein structures on each  $\mathcal{Y}_{m,t}$. 
 The addition of the functions $\eta_i$ prevents $\phi$ from being homogeneous with respect to the $\lambda_r$ action but as $r\rightarrow\infty$  the functions $\eta_i(r^i z )$ are supported on  smaller and smaller \nbhd of origin so $\eta_i(r^i z)/r^{2\alpha}$ go to zero and so we get the asymptotic homogeneity of $\psi$, i.e.: \bq\label{asshomog}\lim_{r\rightarrow \infty} \frac{\psi\circ \lambda_r}{r^{2\alpha}}=\xi.\eq

Since the fibers $\cal{Y}_{m,t}$ are noncompact, existence of parallel transport maps for the fiber bundle $\chi|_{\mathrm{Conf}^0_{2m}}$ is not guaranteed.
 Let $\beta$ be a curve in $\mathrm{Conf}_{2m}$ and $H_\beta $ be the horizontal lift of $\dot{\beta}$ and $Z_{\beta(s)}$ be the projection of $\nabla\psi(\beta(s))$ to $\mathcal{Y}_{m,\gamma(s)}$. Seidel and Smith obtain a \textit{rescaled parallel transport} map $h^{res}_\beta : \mathcal{Y}_{m,\beta(0)}\rightarrow \mathcal{Y}_{m,\beta(1)} $ which is given by integrating the vector field \bq\label{tranportvec}H_\beta -\sigma Z_\beta \eq and then composing by the time $\sigma$ map of $Z_{\beta(1)}$  where  $\sigma$ is a positive constant (depending on $\beta$).
$h^{res}_{\beta}$  is a symplectomorphism defined on arbitrarily large compact subsets of $\cal{Y}_{m,\beta(0)}$. For this procedure to work, one needs the fiberwise critical point set of $\psi$ to project properly under $\chi$ and the homogeneity property \eqref{asshomog} ensures this. See \cite{SS} Section 5A. \\

If $\mu\in \mathbb{C}^{2m}/\mathbf{S}_{2m}$ has only one element of multiplicity two or higher, which we denote by $\mu_1$, denote by $\cal{D}_{m,\mu}$ the set  of singular elements  of $(\chi^{-1}(\mu)\cap \mathcal{S}_m)$ i.e.
\bq\label{Ddef}\cal{D}_{m,\mu}=(\chi^{-1}(\mu)\cap \mathcal{S}_m)^{sub,\mu_1}.\eq
 Let $\cal{D}_m$ be the union of all these $\cal{D}_{m,t}$ regarded as a subset of $\mathcal{S}_m$. It inherits a Kahler metric from $\mathcal{S}_m$.  We have the map $\chi: \cal{D}_{m}\rightarrow \mathbb{C}\times \mathbb{C}^{2m-2}/\mathbf{S}_{2m-2} $. By forgetting the first eingenvalue, the image of $\chi$ can be identified with $\mathrm{Conf}_{2m-2}$.
Lemma \ref{2dkernel} tells us that if $\mu_1=0$ then $\cal{D}_{m,\mu}$ can be identified with $\mathcal{Y}_{m-1,\mu \backslash \{\mu_1 \}}$.
It can be shown (\cite{SS} Section 5A) that $\chi $ is a differentiable fiber bundle and we have rescaled parallel transport maps
\bq\label{paralleltransD} h^{res}_\beta: \cal{D}_{m,\beta(0)} \rightarrow \cal{D}_{m,\beta(1)} \eq
for any curve $\beta$ in $\mathrm{Conf}_{2m-2}$. These parallel transport maps are compatible with those for $\cal{Y}_{m-1,\mu \backslash \{0\}}$ under the identification above provided that $\beta$ lies in $\mathrm{Conf}_{2m-2}^0$. This is because of the special (product) form of the symplectic structure.\\

\subsection{Lagrangian submanifolds from crossingless matchings}\label{crossless}

Let $\mu\in \mathrm{Conf}_{2m}$. A \textit{crossingless matching} $D$ with endpoints in $\mu$ is a set of $m$ disjoint embedded curves $\delta_1,\ldots,\delta_m$ in $\mathbb{C}$ which have (only) elements of $\mu$ as endpoints.  See Figure \ref{crossinglessmatchings}. To $D$ we associate a Lagrangian submanifold $L_D$ of $\cal{Y}_{m,\mu}$ as follows. Let 
$\{\mu_{2k-1},\mu_{2k}\}\subset \mu$  be the endpoints of $\delta_k$ for each $k$. Let $\gamma$ be a curve in $\mathrm{Conf}^0_{m}$ such that  $\gamma(t)=\{\gamma_1(t), \gamma_2(t), \mu_3, \mu_4,\ldots,\mu_{2m}\}$, $\gamma_i(0)=\mu_i, i=1,2$ and as $s\rightarrow 1$, $\gamma_1(t), \gamma_2(t)$ approach each other on $\delta_1$  and collide. For example if $\delta_1(t)$ is a parametrization of $\delta_1$ s.t $\delta_1(0)=\mu_1,\delta_1(1)=\mu_2$ the we can take $\gamma(t)=\{\delta(t/2), \delta(1-t/2), \mu_3,\ldots,\mu_{2m}\}$. Set $\bar{\mu}=\mu \backslash \{\mu_{1},\mu_{2}  \}$, $\mu'=\gamma(1)$.

If $m=1$ then relative vanishing cycle construction for $\chi: \cal{S}_1\rightarrow \bb{C}$ with the critical point over $\gamma(1)=0$ gives us  a Lagrangian sphere $L$ in $\cal{Y}_{1,\gamma(1-s)}$ for small $s$. Using reverse parallel transport along $\gamma$ we can move $L$ to $\mathcal{Y}_{1,\mu}$ to get our desired Lagrangian submanifold.  Now for arbitrary $m$ assume by induction that  we have obtained a Lagrangian $L_{\bar{D}}\subset \cal{Y}_{m-1,\bar{\mu} } $ for  $\bar{D}$ which is obtained from $D$ by deleting $\delta_1$.
Now $\cal{Y}_{m-1,\bar{\mu}}$ can be identified with  $D_{m,\tau}$ where $\tau=(0,0, \mu_3-(\mu_1+\mu_2)/(2m-2),\ldots,  \mu_{2m}-(\mu_1+\mu_2)/(2m-2))$. Use parallel transport  to move $L_{\bar{D}}$ to $D_{m,\gamma(1)}$ . The later one is the set of singular points of $\cal{Y}_{m,\gamma(1)}$  so Lemma \ref{2nbhd} tells us that we can use relative vanishing cycle construction for $L_{\bar{D}}$ to obtain a Lagrangian in $\cal{Y}_{m,\gamma(1-s)}$ for small $s$. Parallel transporting  it along $\gamma$ back to $\cal{Y}_{m,\mu}$ we obtain our desired Lagrangian which is topologically a trivial sphere bundle on $L_{\bar{D}}$. We see that $L_D$ is  diffeomorphic to a product of spheres. Different choices of the curve $\gamma$ result in Hamiltonian isotopic Lagrangians. The same holds if we isotope the curves in $D$ inside $\mathbb{C}\backslash \mu$.\\

 \begin{figure}[ht]
\includegraphics[width=4.5in]{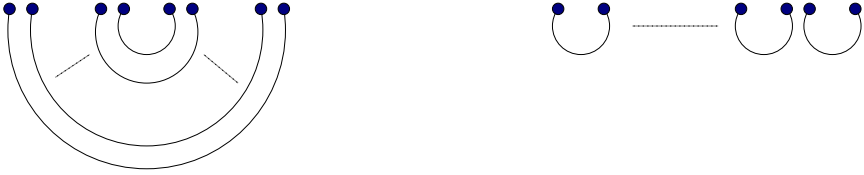}
\caption{Two crossingless matchings}
\label{crossinglessmatchings}
\end{figure}

\subsection{Grading}\label{grading}
The topic of this subsection is the absolute grading of Floer cohomology groups.
 We discuss the special case which is of concern in this paper, i.e. the case of $\mathbb{Z}$ grading. The reader is referred to \cite{gradedlag} for more detail. Let $(M^{2n},\omega)$ be a symplectic manifold. We know that for a Lie group $G$ and its maximal compact subgroup $K$, the isomorphism classes of $G$ bundles on a topological space $X$ are in one-to-one correspondence with those of $K$ bundles over $X$. The symplectic group $Sp(2n)$ has $U(n)$ as maximal compact subgroup therefore  upon choosing a compatible almost complex structure the structure group of $TM$ can be reduced to $U(n)$.
 Denote by $P$ the principal  $U(n)$ bundle we get from  $TM$ in this way.
\begin{definition}
 An \textit{infinite Maslov cover} of $M$ is a $SU(n)$ bundle $\tilde{P}$ over $M$ 
 such that $\tilde{P}\times_{SU(n)} \mathbb{C}^n=TM$.
 \end{definition}
 Since the universal cover \smash{$\widetilde{U(n)}$} of $U(n)$ has $SU(n)$ as maximal compact subgroup, we get an equivalent definition by replacing $SU(n)$ with \smash{$\widetilde{U(n)}$} in the above definition.

\begin{lemma} 
If the structure group of $TM$  can be further  reduced to $SU(n)$ then $M$ has an infinite Maslov covering.\end{lemma}
\begin{proof}
  Let $g_{\alpha,\beta}$ and $g'_{\alpha,\beta}$ be transition functions of $TM$ as a $U(n)$ resp. $SU(n)$ bundle over the same covering of $M$. Then the transition functions $\{(g_{\alpha,\beta},t)|\quad g_{\alpha,\beta}=e^{it} g'_{\alpha,\beta} \}$ define a \smash{$\widetilde{U(n)}$} bundle $\tilde{P}$ over $M$ for which we have $\tilde{P}\times_{\widetilde{U(n)}} \mathbb{C}^n=TM$.

 \end{proof}
The isomorphism classes of such covers are in one to one correspondence with $H_1(M)$.
 Let $\cal{L}\rightarrow M$ be the bundle whose fiber at $x\in M$ is the Lagrangian Grassmanian of $T_{x}M$. A Maslov cover $\tilde{P}$, induces a covering $\tilde{\cal{L}}\rightarrow\cal{L}$ given by  \smash{$\tilde{P}\times_{SU(n)} \widetilde{Lag}(n)$}.
  Each Lagrangian submanifold $L$ of $M$ determines a section $s_L$ of $\cal{L}$. A \textit{grading} of $L$ is a cover $\tilde{s}_L$ of $s_L$.
If $L_0, L_1$ are two Lagrangian submanifolds of $M$ that intersect transversely we can assign an absolute grading to each intersection point $x$ as follows. Let $\tilde{\Lambda_i}=\tilde{s}_{L_{i}}$ for $i=0,1$. Let $\tilde{\Lambda}(t)$ be a path joining $\tilde{\Lambda_0}$ to $\tilde{\Lambda_1}$ and let $\Lambda(t)$ be its projection. Define \bq\label{maslovgrade}\on{deg}(x)=\mu^{path}(\Lambda, T_x L_0)+\frac{n}{2} \eq  where $\mu^{path}$ is the Maslov index for paths.
It does not depend on the choice of the liftings  $\tilde{s}_{L_{i}}$ because of the naturality of $\mu^{path}$. 
In general if the canonical line bundle is not trivial, one can obtain only a $\mathbb{Z}/N$ grading for some $N\in \mathbb{N}$.

There is an equivalent way of describing this grading.  If the canonical bundle of $M$ is trivial, it has a global section (or trivialization) $\eta$.
  The global section $\eta$ gives us a map $\alpha_M :M\rightarrow S^1$ given by \bq\label{phaseofM}\alpha_M(x)=\frac{\eta(u_1,\ldots,u_n)^2}{|\eta(u_1,\ldots,u_n)|^2}\eq

  for any orthonormal basis $u_1,\ldots,u_n$ of $T_xM$.
  For each Lagrangian submanifold $L$ we can define a \textit{phase function} $\alpha_L:L\rightarrow S^1$  by \bq\label{phase}\alpha_L(x)=\frac{\eta(v_1,\ldots,v_n)^2}{|\eta(v_1,\ldots,v_n)|^2}\eq  for any orthonormal basis $v_1,\ldots,v_n$ of $T_xL$.
   \begin{altdef}
   A \textit{grading} of $L$ is a choice of a real valued function $\tilde{L}$ such that $\alpha_L=\exp(2\pi i\tilde{L})$.
   \end{altdef}
   For a pair of transversely intersecting graded Lagrangians $L_0, L_1$ and $x\in L_0\cap L_1$ one can set \bq\label{degdef2}\on{deg}(x)=\frac{n}{2}+\tilde{L_0}(x)-\tilde{L_1}(x).\eq
    We denote by $L\{m\}$, $L$ with its grading shifted by $m$, ie. $\tilde{L}\{m\}=\tilde{L}-m$.
  A grading for a diffeomorphism $\phi$ from $M$ to itself is a choice of a function $\tilde{\phi}:M\rightarrow \mathbb{R}$ such that $\exp(2\pi i\tilde{\phi}(x))=\alpha_M(x)/\alpha_M(\phi^{-1}(x))$. $\phi(L)$ has a preferred grading given by \bq\label{phiL}\widetilde{\phi(L)}(x)=\tilde{L}(\phi^{-1})(x)+\tilde{\phi}(\phi^{-1}(x)).\eq
A choice of grading for $\phi$ induces a grading on the graph $\Gamma$ of $\phi$: \bq\label{gradinggraph} \tilde{\Gamma}(x,\phi(x))=\tilde{\phi}(x). \eq

  Since each manifold $\cal{Y}_{m,\nu}$ is a   submanifold of the affine space $\Sm$ and has trivial  normal bundle, its Chern classes are zero. This together with the fact that $H_1(\cal{Y}_m)=0$ implies that   the canonical bundle of  $\cal{Y}_{m,\nu}$ is trivial and so has a unique infinite Maslov cover. We start by choosing global sections $\eta_{\Sm}$ and $\eta_{\frak{h}/W}$. Then we choose trivializations for regular fibers of $\chi_{\Sm}$  characterized by $\eta_{\cal{Y}_{m,t}}\wedge \chi^*\eta_{\frak{h}/W}=\eta_{\Sm}.$ If we choose a grading for $L\subset \cal{Y}_{m,t_0}$ and $\beta$ is a curve in $\mathrm{Conf}_{2m}$ starting at $t_0$, one can continue the grading on $L$ uniquely to $h_{\beta|_{[0,s]}}(L)$ for any $s$. Therefore the grading of $L$ uniquely determines that of $h_\beta(L)$. 

\subsection{The invariant}

Now we can define the Seidel-Smith invariant. The definition uses Floer cohomology which we review (in a bit more general setting) in section \ref{floercoho}.
Let  $\cal{D_+}$ be the crossingless matching at the left hand side of picture \ref{crossinglessmatchings}. If a link $K$ is obtained as closure of a braid $\beta\in Br_m$, let $\beta'\in \mathrm{Conf}_{2m}$ be the braid obtained from $\beta$ by adjoining the identity braid $\on{id}_m$.
\begin{definition}\label{defss}
$$\mathpzc{Kh}^{*}_{sympl}(K)=HF^{*+m+w}(L_{\cal{D_+}}, h^{res}_{\beta'}  (L_{\cal{D_+}}))$$
\end{definition}
Here  $w$ is the writhe of the braid presentation, i.e. the number of positive crossings minus the number of the negative crossings in the presentation.
Since the manifold is convex at infinity and the Lagrangians are exact, the above Floer cohomology is well-defined.
Independence from choice of $\beta$ is proved in \cite{SS}, section 5C. Well-definedness of the invariant developed in the present paper gives an alternative proof. See Theorem \ref{wellfunc}.

\section{ Stein symplectic valued field theories}\label{TSFT}
In this section we generalize symplectic valued field theories to allow a class of Lagrangian submanifolds of Stein manifolds. 

\subsection{Some remarks on Lagrangian submanifolds of Stein manifolds}\label{laginstein}
Let $(M,\psi)$ be a Stein manifold. This means that $\psi$ is proper, bounded below and  $-d d^c \psi$ is a symplectic form on $M$. For a subset $N\subset M$, we denote by $N_{\leq c}$ ($N_c$) the intersection of $N$ with sublevel (level) sets of $\psi$. Also set $\theta=-d^c \psi.$
\begin{definition}
 A Lagrangian submanifold $L\subset M$ is called $c$-\textit{allowable}  if it is exact and $\psi|_L$ does not have any critical points in $M_{\geq c }$. It is called \textit{allowable} if it is $c$-allowable for some $c$.
 In addition any compact monotone Lagrangian submanifold of  
  any symplectic manifold is considered allowable.
\end{definition}
 Note that this implies that $L$ intersects the level sets of $\psi$ transversely at infinity. Compact Lagrangian submanifolds  of  Stein manifolds are evidently allowable.
Let $Z=\nabla \psi$ be the Liouville vector field on $M$. 
Denote by $\kappa:\mathbb{R}\times M\rightarrow M$ 
the flow of $Z$. 
\begin{definition}
A Lagrangian  $L$ in a Stein manifold $M$  is said to  have a \textit{conical end} if there is a constant $c$ such that $L$ intersects $M_c$ transversely and $\kappa([0,\infty)\times (L\cap M_c))$ equals $L_{\geq c}$.
\end{definition}
An exact Lagrangian with conical end is clearly allowable. Our next task is to assign a Lagrangian with conical end to an allowable Lagrangian which can replace the former when  Floer cohomology is concerned. We first need some definitions. 

\begin{definition}
 A Lagrangian submanifold $\Lambda$ of  the symplectic manifold $M_{\leq c}$ is called $\kappa$-\textit{compatible} if it intersects $M_c$ transversely (possibly empty) and there is an $\epsilon>0$ such that $\kappa \left([-\epsilon,0]\times \Lambda_c \right)=( \, _{c-\epsilon \leq} \Lambda_{\leq c}\,)$.
\end{definition}
\begin{definition}
A  Hamiltonian isotopy induced by a time-dependent function $H_t$ is called \textit{conical} if there is a constant $c$ such that $ H_t\circ \kappa(r,x)= r H_t(\kappa(0,x)) $ for all $r\geq 0$, $t\in [0,1]$ and $x$ in $\,M_{\geq c}$.
\end{definition}
Let $\phi_t$ be a one parameter family of symplectomorphisms of $M$ and $L\subset M$ a Lagrangian submanifold. The isotopy $\phi_t(L)$ is called \textit{exact} if $\phi_t^* \theta|_{L}=\theta|_{L}+d K_t$  for any $t\in [0,1]$ where $K_t$ is a function on $L$ depending smoothly on $t$.
We have the following facts from  \cite{KhS} Section 5.
\begin{lemma}
i) Any exact Lagrangian in $M_{\leq c}$ which intersects the boundary transversely can be exact-isotoped rel boundary in $M_{\leq c}$ to a $\kappa$-compatible one.

ii)If $\Lambda_t$ is a Lagrangian isotopy in $M_{\leq c}$  such that all $\Lambda_t$ intersect the boundary transversely and $\Lambda_0, \Lambda_1$ are $\kappa$-compatible then there is another isotopy $\Lambda'_t$ with the same endpoints such that $\Lambda_t\cap M_{c}=\Lambda'_t\cap M_c$ for any $t\in [0,1]$ and all $\Lambda_t$ are $\kappa$-compatible. If $\Lambda_t$ is exact, $\Lambda'_t$ can be chosen to be so.
\end{lemma}
We include the proof for completeness.
\begin{proof}{\cite{KhS}, Lemma 5.2}
i)Let $L$ be such a Lagrangian. Choose $r<c$ such that 
$K=L_{\geq -r}$ deformation retracts onto $\partial L$. $\theta|_K$ is closed and vanishes on the boundary
so $\theta|_K=df$ for some $f$ on $K$ which vanishes on boundary. Extend $f$ to a smooth function $h$ on $M_{\leq c}$ vanishing on the boundary and set $\theta_t=\theta-tdh$ for $t\in [0,1]$. Let $Z_t$ be a vector field such that $\iota_{Z_t} \omega=\theta_t$. Since $d\theta_t=\omega$, $Z_t$ is symplectic and defines an embedding $\kappa_t:M_c\times [0,\infty]\rightarrow M_{\leq c}$. Let $U$ be a \nbhd of $M_c$ which lies in the image of $k_t$ for all $t$ and deformation retracts onto $M_c$. Set $X_t=\partial \kappa_t/ \partial t\circ (\kappa_t)^{-1}|_U$. It is a symplectic vector field which vanishes on $\partial M$. $\iota_{X_t}\omega$ is closed and vanishes on boundary so equals $dH_t$ for some $H_t$ on $U$ vanishing on $\pi_c$.
Each $X_t$ can be extended to a Hamiltonian vector field $\tilde{X}_t$ on $M_{\leq c}$. Let $\phi_t$ be the flow of the time dependent vector field $(\tilde{X}_t)$. We have $\phi_t\circ \kappa_0=\kappa_t$  on a \nbhd of $M_c$ so
 $\phi_1^*\theta=({\kappa_0}^{-1})^* {\kappa_t}^*\theta=\theta_1$ near $M_c$. Therefore \bq\label{yyy}\phi_1^*(\theta|_L)=\theta_1|_L=0\eq near the boundary of $L$ and so $\phi_1(L)$ is $\kappa$-compatible. By \eqref{yyy}, $(\phi_t(L))$  is exact.\\

ii) It is just a parameterized version of i).

\end{proof}

\begin{lemma}
Any exact Lagrangian  isotopy  rel boundary in $M_{\leq c}$ which consists only of $\kappa$-compatible Lagrangians can be embedded into a Hamiltonian isotopy.
\end{lemma}
The proof is similar to the proof of the fact that a symplectomorphism of zero flux is Hamiltonian.

\begin{definition}
 Let $L$ be a $c$-allowable Lagrangian in $M$, we take $\Lambda=L_{\leq c}$, isotope $\Lambda$ to a $\kappa$-compatible $\Lambda'$ and denote by $\mathfrak{C}_c(L)\subset M$ the Lagrangian with conical end associated to $\Lambda'$. This means that $\mathfrak{C}_c(L)_{>c}$ is the image of $L_c$ under the Liouville flow.
 \end{definition}
 It follows from the above lemmas that $\mathfrak{C}_c(L)$ is well-defined up to conical Hamiltonian isotopy. Moreover if $L_{\leq c}$ and $L'_{\leq c}$ are exact isotopic rel boundary then $\mathfrak{C}_c(L)$ and
$\mathfrak{C}_c(L')$ are conical Hamiltonian isotopic.

\subsection{Symplectic valued topological field theories}\label{tsfts}

A \textit{Lagrangian correspondence} between two symplectic manifolds $M_0$ and $M_1$ is a Lagrangian submanifold of $M^-_0\times M_1$. If $L_{0,1}$ is a Lagrangian correspondence between $M_0,M_1$ and $L_{1,2}$ is a correspondence between $M_1,M_2$ then the \textit{composition} $L_{0,1}\circ L_{1,2}$ is defined as
$$L_{0,1}\circ L_{1,2}:=\{(m,m'') |\exists m'\in M_1 s.t. (m,m')\in L_{0,1} , (m',m'')\in L_{1,2} \}$$
which is a subset of $M_0\times M_1.$

\begin{definition}
 This composition is \textit{embedded} if $L_{0,1}\times L_{1,2}\subset M_0\times M_1\times M_1\times M_2$ intersects  the diagonal $M_0\times \Delta_{M_1}\times M_2$ transversely and the projection $\pi_{0,2}$ embeds the intersection into $M_0\times M_2$. In this case the composition is a Lagrangian submanifold of $M^-_0\times M_2$.
\end{definition}

If $M_i$, $i=0,1,2$ have trivial canonical bundle and we have chosen  trivializations $\eta_{M_i}$ for each $i$ then gradings $\tilde{L}_{0,1}$ and  $\tilde{L}_{1,2}$ determine a grading on $L_{0,2}=L_{0,1}\circ L_{1,2}$, with regard to the trivialization $\eta_{M_0}\wedge\eta_{M_2}$, given by \bq\label{composegrade} \tilde{L}_{0,2}(m,m'')=\tilde{L}_{1,2}(m,m')+\tilde{L}_{0,1}(m',m'')\eq where $m'$ is the unique point such that $(m,m',m',m'')\in \left(L_{0,1}\times L_{1,2}\right)\cap M_0\times\Delta_{M_1}\times M_2$  provided that the composition is embedded. If the $M_i$ are Stein and the   Lagrangian correspondences  have conical ends then it is easy to see that their composition has a conical end as well. Therefore we have the following.
\bq\label{Cofcomposition}
\mathfrak{C}_c(L_{0,1})\circ \mathfrak{C}_c(L_{1,2})\simeq \mathfrak{C}_c(L_{0,1}\circ L_{1,2})
\eq

\begin{definition}
A \textit{generalized Lagrangian correspondence} between symplectic manifolds $M,M'$ consists of a sequence $M=M_0, M_1,\cdots, M_n=M'$ of symplectic manifolds and a sequence $\underline{L}=(L_{0,1},\cdots  L_{i,i+1},\ldots L_{n-1,n})$  such that $L_{i,i+1}$  is a Lagrangian correspondence between $M_i$ and $M_{i+1}$. Generalized Lagrangian correspondences can be composed by concatenation. We call $\underline{L}$ \textit{compact} if each $L_{i,i+1}$ is compact. It is called 
\textit{allowable} if 
all the  $L_{i,i+1}$ are so.
\end{definition}
\begin{definition}
 We denote the composition (concatenation) of $\underline{L}$ and $\underline{L}'$  by $\underline{L}\#\underline{L}'$. Suppose all the manifolds $M_i$ involved have trivial canonical bundle and we have chosen trivializations $\eta_{M_i}$ for each $M_i$. A \textit{grading} on $\underline{L}$ is a lift $\tilde{L}_{i,i+1}$ of $\alpha_{L_{i,i+1}}$ for each $i$ where the phase functions $\alpha_{L_{i,i+1}}$  are with respect to $\eta_{M_i^-}\wedge\eta_{M_{i+1}}$.
\end{definition}
 Assume we have chosen a trivialization $\eta_{M}$ for  the canonical bundle of $M$. The canonical  bundle of $M ^-$ is the dual of $\bigwedge^{\dim M}TM$. Thus we can take $\sqrt{-1}\eta^{-1}$ as a trivialization of the determinant bundle of $M^-$. So the phase function of a Lagrangian $L\subset M$ is the negative of the inverse of the phase function of the same Lagrangian as a subset of $M^-$. (We denote the later one by $L_-$.)  Therefore a grading $\tilde{L}$ of $L$ as a Lagrangian submanifold of $M$ induces, in a natural way, a grading of $L$ as a  submanifold of $M^-$. This new grading equals ${k}/{2}-\tilde{L}$ for some integer $k$. We choose $k=n=\on{dim} L$ therefore we have
 \bq\label{neggrade} \tilde{L}_-=\frac{n}{2}-\tilde{L}. \eq

According to \cite{WW} Section 2.2, the  \textit{symplectic category} is the category whose objects are compact monotone 
symplectic manifolds (including exact ones)
 and whose morphisms are equivalence classes of compact generalized Lagrangian correspondences.  The equivalence relation on morphisms is generated by the following two relations.
Firstly, $$(L_0,L_{0,1},\ldots, L_{i,i+1},\ldots, L_{n-1,n},L_n)$$ is equivalent to $$(L'_0, L'_{0,1},\ldots, L'_{i,i+1}, \ldots, L'_{n-1,n},L'_n)$$ if each $L_{i,i+1}$ is Hamiltonian isotopic to $L'_{i,i+1}$ in $M_i^-\times M_{i+1}$. Secondly $$(L_0, L_{0,1},\ldots, L_{i-1,i}, L_{i,i+1},\ldots, L_n)$$ is equivalent to $$(L_0,L_{0,1},\ldots, L_{i-1,i}\circ L_{i,i+1},\ldots, L_n)$$ whenever the composition
$L_{i-1,i}\circ L_{i,i+1}$ is  embedded. The idea of symplectic category goes back to Weinstein \cite{weisteincat} where he considered only Lagrangian correspondences as morphisms. Since the composition of two Lagrangian correspondences might not be embedded, he did not obtain a genuine category.

\begin{definition}
A $d+1$ dimensional \textit{symplectic valued topological field theory} is a functor from the category ($d$-manifolds, cobordisms) to the symplectic category. Here ``cobordism'' means cobordism modulo isotopy.
\end{definition}

In the present paper we need to include a class of noncompact symplectic manifolds and Lagrangians so we  enlarge  the 
symplectic category.
 \begin{definition}
The \textit{noncompact symplectic category} has  objects of the symplectic category plus  Stein manifolds 
  as objects. Morphisms are given by equivalence classes of generalized Lagrangian correspondences\\ $(L_{0,1},\ldots, L_{n-1,n})$  such that each $L_{i,i+1}$ is 
allowable. The equivalence relation on correspondences is the one in the symplectic category with the difference that  we restrict  the first kind of equivalence of morphisms to include only 
exact isotopies.
\end{definition}

\subsection{Floer cohomology}\label{floercoho}

Let  $\underline{L}$ be a generalized Lagrangian correspondence as in the last section. By adding a trivial Lagrangian correspondence (i.e. the diagonal)  if necessary we can assume that $n=2k+1$ is odd. Define \bq\label{L0}\mathcal{L}_0=L_0\times L_{1,2}\times \cdots L_{2k-1,2k}\eq and
\bq\label{L1}\mathcal{L}_1=L_{0,1}\times L_{2,3} \times \cdots\times L_{2k+1}\eq which are Lagrangian submanifolds of $M=M_0^-\times M_1\times \cdots M_n^-$.
One can, under good circumstances, associate to $\underline{L}$ the Floer cohomology group

\bq\label{HF}HF(\underline{L}):=HF(\mathfrak{C}_c(\cal{L}_0),\mathfrak{C}_c(\cal{L}_1))\eq
for some $c$ large enough. Here we briefly review the construction emphasizing the aspects that are of concern in this paper. See \cite{WW} for more details in the case of compact Lagrangians. First we assume that there is a conical Hamiltonian isotopy of $M$ which makes $\cal{L}_0$ and $\cal{L}_1$ intersect transversally at finitely many points. We call this assumption \textit{finite intersection} of the Lagrangians. It holds if one of the Lagrangians is compact. More generally it holds if one of the correspondences $L_{i,i+1}$ is compact and all the others are proper in the following sense.
\begin{definition}\label{propercorr}
A correspondence $L\subset M^-\times M'$ is called \textit{proper} if for each point $y\in M'$ the set $\{x\in M | (x,y)\in M'\}$  is compact.

\end{definition}

We denote the isotoped Lagrangians by the same notation.
Let $J_i$ be a compatible almost complex structure on $M_i$. If $M_i$ is Stein then we require $J_i$ to be invariant under the flow of Liouville vector field outside a compact set. We call such an almost complex structure \textit{asymptotically invariant}.   From the $J_i$ we get an almost complex structure $$J=(-J_0, J_1,\ldots, (-1)^{n-1} J_n)$$ on $M$.
Let $x,y\in \mathcal{L}_0 \cap  \mathcal{L}_1$ and let $J_t$ be a one parameter family of almost complex structures on $M$ for which there is an $r_0>0$ such that $J_t=J$ outside $M_{\leq r_0}$ for all $t$. Denote by $\cal{M}(x,y)$ the moduli space of the strips $$u:\mathbb{R}\times [0,1]\rightarrow \cal{Y}$$ such that
$\frac{\partial}{\partial t} u(s,t)=J_t \frac{\partial}{\partial s} u(s,t) $ and $u(s,i)\in \cal{L}_i$ for $i=0,1 $  and $u(-\infty,t)=x, u(\infty,t)=y$.
 $\mathbb{R}$ acts on  $\cal{M}(x,y)$ by shifting the parametrization.\\

Consider the abelian group $CF(\mathcal{L}_0,  \mathcal{L}_1)$ generated freely by the intersection points of the two Lagrangians.
The following differential  makes $CF(\mathcal{L}_0,  \mathcal{L}_1)$ into a chain complex and the Floer cohomology $HF(\mathcal{L}_0,  \mathcal{L}_1)$  is defined to be the cohomology of this complex.  \\
\bq\label{dfloer} \partial x=\sum_{y\in \mathcal{L}_0\cap  \mathcal{L}_1} \#\cal{M}_1(x,y)/\mathbb{R}\eq
Here $\cal{M}_1(x,y)$ is the one dimensional part of the moduli space. The count is a priori in $\mathbb{Z}/2$. To be able to define Floer cohomology groups as $\mathbb{Z}$ one needs coherent orientations on the moduli spaces cf. \cite{FOorientation}.
 For this invariant to be well-defined one has to take care of the following issues: transversality of moduli spaces, nonexistence of bubbling, compactness and invariance under Hamiltonian isotopy. In the following discussion we assume that the first two criteria hold and focus on the last two. First compactness.

\begin{lemma}\label{compact}
Assume $(M,\psi)$ is a Stein manifold,  $\cal{L}_0,\cal{L}_1$ are $C$-allowable Lagrangians for some $C\geq r_0$ and that $M_{\leq C}$ contains the intersection points of the Lagrangians. Then for any Riemann surface $S$ with boundary and any $J$-holomorphic map $u:S\rightarrow M$ with $u(\partial S)\subset \cal{L}_0 \cup\cal{L}_1$, the image of $u$ lies in  $M_{\leq C}$ (which is independent of $S$ and $u$).
\end{lemma}

\begin{proof} {\cite{EHS}, \cite{OHnoncompact}}
For $u\in \cal{M}(x,y)$, $\psi \circ u$ cannot have a maximum on the interior of the curve outside $M_{\leq C}$ by the maximum principle. This is because $u$ is $J$ holomorphic outside $M_{\leq r_0}$. Assume it has a maximum on a boundary point $p$ 
 ie. $\max \psi \circ u=\psi \circ u(p)=R>C $.
We can pick holomorphic coordinates on $S$ in a \nbhd of $p$ and therefore regard $u$, in a \nbhd of $p=(s_0,1)$,  as a function on some rectangle $[l_1,l_2]\times [1-\delta, 1]$.
  We have $d\psi (\frac{\partial}{\partial s} u)(s_0,1)=0$. So $\frac{\partial}{\partial s} u(s_0,1)\in T\cal{L}_1 \cap TM_R$. By assumption the intersection is transverse therefore it is Legendrian so $d\psi(J\frac{\partial}{\partial s} u(s_0,1))=0$. But we have $\frac{\partial}{\partial t} u=J \frac{\partial}{\partial s} u$ so $d\psi (\frac{\partial}{\partial t} u)(s_0,1)=0$. This contradicts the strong maximum principle which implies $d\psi (\frac{\partial}{\partial t} u)>0$.
\end{proof}

This enables us to apply the rescaling  argument to show that the limit of a bounded energy sequence of such curves is either a broken trajectory or a  curve with sphere or disc bubbles. Therefore we can compactify $\cal{M}(x,y)$ by adding these limiting curves to it.  If in addition both $M$ and the Lagrangians are  exact, no bubbling occurs and so  the sum in \eqref{dfloer} is finite and we get $\partial^2=0$.

An important property of Floer cohomology in compact manifolds is invariance under Hamiltonian isotopy. However this invariance does not hold for isotopies  with noncompact support. We have the following theorem of Oh.

\begin{theorem}[Oh \cite{OHnoncompact}]
Let $M$ be a Stein manifold 
 and $L,L'$ Lagrangian submanifolds of $M$, $H_t$, $t\in[0,1]$, a time-dependent conical Hamiltonian  on $M$ and $L_t$ the image of $L$ under time $t$ map of $H$. Assume $\cup_t L_t\cap L'$ is compact.\\
\\
 i) We can find another Hamiltonian $H'$ such that 
the time one map of $H'$ equals that of $H$ and $H'$ is positive on $\cup_{s\in [\delta,1-\delta]}  L\cap L_s$ for some $0<\delta<1$.\\
 ii) If in addition $H'$ is conical  then there is a canonical isomorphism\\ $HF(L_0,L')\cong HF(L_1,L')$.

\end{theorem}

\begin{proposition}\label{floerwell}
  Let  $L,L'$ be two Lagrangians in a  Stein manifold $(M,\psi)$ which are $C$-allowable  and satisfy  the finite intersection condition. Then the Floer cohomology $HF(\mathfrak{C}_c(L),\mathfrak{C}_c(L'))$ is well-defined and is invariant under conical Hamiltonian isotopies. It is independent of $c$ when $c> C$ and $L\cap L'\subset M_{\leq c}$.
\end{proposition}
\begin{proof}It follows from the  last two results along with those in section \ref{laginstein}.
\end{proof}

In order for Floer cohomology to define a well-defined map on the symplectic category, we must understand the effect of composition of Lagrangian correspondences on Floer cohomology. The following important \textit{Functoriality Theorem} is proved in \cite{WW}.

\begin{theorem}[Wehrheim, Woodward \cite{WW}]\label{fctr}
Let $\underline{L}=(L_{0,1},L_{1,2},\ldots,L_{n-1,n})$ be a generalized Lagrangian correspondence between manifolds $M_0,\ldots, M_n$ such that for some $0<j<n$ the composition $L_{j,j+1}\circ L_{j-1,j}$ is embedded. Denote $$ \underline{L}'=(L_0,\ldots,L_{j,j+1}\circ L_{j-1,j},\ldots,L_n).$$ Assume the $M_i$ are compact and monotone with the same monotonicity constant and all  $L_{i,i+1}$ as well as   \eqref{L0} and \eqref{L1} are monotone. 
Assume in addition that each $L_{i,i+1}$ is oriented, relatively spin, graded and  its minimal Maslov number is at least three.  
Then with the induced grading and relative spin structure on $L_{i,i+1}\circ L_{i-1,i}$ there is a canonical isomorphism  
\bq\label{functorialityeq} HF(\underline{L})\cong HF(\underline{L}')\eq of graded abelian groups.
\end{theorem}

\begin{proposition}
 If the $M_i$ are Stein and $L_{i,i+1}$ are allowable, relatively spin  and satisfy the finite intersection condition then \eqref{functorialityeq} holds.

\end{proposition}

\begin{proof}
It is easy to see that the generators for the two Floer groups are in one-to-one correspondence.  
Take $x,y\in \cal{L}_0\cap\cal{L}_1$ and let $\cal{M}_\delta(x,y)$ be the   moduli space of pseudoholomorphic strips $u=(u_0,\ldots u_n)$ where $$u_i:[0,1]\times \mathbb{R}\rightarrow M_i$$ if $i\neq j$ and $u_i:[0,\delta]\rightarrow M_j$ 
 and with boundary condition for $u$ given by $\underline{L}$ and with $x$ and $y$ as asymptotic points.
Let $ \cal{M}'(x,y)$ be the moduli space of strips  $$(u_0,\ldots,u_{i-1},u_{i+1},\ldots,u_n):[0,1]\times \mathbb{R}\rightarrow M_0\times \cdots M_n$$ 
with boundary condition $\underline{L}'$ and asymptotic points $x,y$. Since the $M_i$ are Stein and the Lagrangians are allowable and satisfy the finite intersection property then Lemma \ref{compact} implies that the holomorphic curves  in $\cal{M}_\delta(x,y)$ and $ \cal{M}'(x,y)$ stay in a compact submanifold of $M_0\times \cdots\times M_n$ (which doesn't depend on $\delta$) and so the proof proceeds as in \cite{WW} to show that for $\delta$ small enough the two moduli spaces are bijective.
Note that exactness of the $L_{i,i+1}$ implies monotonicity and rules out bubbling so we do not need the assumption on the minimal Maslov index in \ref{fctr}.
\end{proof}

\begin{remark}
Assume  we have two Hamiltonian isotopic Lagrangian correspondences  $L_0, L_1$   and another correspondence $L'$ such that both compositions $L_0\circ L'$  and $L_1\circ L'$ are embedded. In general we don't know if  $L_0\circ L'$  and $L_1\circ L'$ are Hamiltonian isotopic or not.  However the proof of the Functoriality Theorem implies that Floer cohomology is invariant under such an isotopy.
\end{remark}

\section{A Symplectic Valued Topological Field Theory}

\subsection{Tangles}

A tangle $T$ is defined to be a compact one-dimensional submanifold of (a diffeomorphic image of) $\mathbb{C}\times [0,1]$ such that $i(T):=(T\cap \mathbb{C})\times \{0\}\subset \mathbb{R} \times  \{0\}$ and $t(T):=(\partial T\cap \mathbb{C})\times \{1\}\subset \mathbb{R} \times  \{1\}$. The second assumption makes $i(T)$ and $t(T)$ ordered sets.  In this paper we deal only with tangles with an even number of initial points and end points. If $\# i(T)=2m, \#t(T)=2n$ we say $T$ is an $(m,n)$-tangle and write  $m T n$. We also allow $m$ and/or $n$ to be zero.  
\begin{definition}
Two tangles $T, T'$ are called \textit{equivalent} if there is a continuous family ${T_t}$  of tangles for ${t\in [0,1]}$ such that $T_0=T$ and $T_1=T'$ and the order of $i(T_t)$ and of $t(T_t)$ is fixed.
\end{definition}

 \begin{figure}[ht]
\includegraphics[scale=0.6]{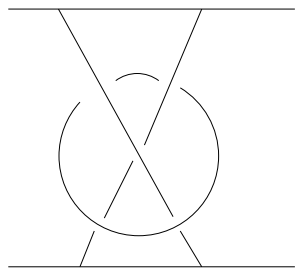}
\caption{A $(1,1)$-tangle}
\label{atangle}
\end{figure}

Two tangles  $T_1, T_2$ can be composed (concatenated) if $t(T_1)=i(T_2)$. Two equivalence classes $[T_1]$ and $[T_2]$ of tangles can be composed if
$\#t(T_1)=\#i(T_2)$ and composition is done using the ordering on $t(T_1)$ and $i(T_2)$. Composition of tangles is denoted by juxtaposition.
We will use the  notation $id_m$, $\cap_{i;m}$, $\cup_{i;m}$ and $\sigma_{i;m}$ for the elementary tangles in figures \ref{elembraids} and \ref{elementangles} where $m$ denotes the number of the strands. We might ignore $m$ when there is no confusion. 
When we say a tangle $T$ is equivalent to, say, $\cap_{i;m}$, we implicitly have a one to one correspondence between $i(T)$ and $\{1, 2,\ldots, 2m-2  \}$ and also between $t(T)$ and $\{1, 2,\ldots, 2m\}$ in mind.\\

A \textit{decomposition} of $T$ is a sequence of tangles 
\bq\label{tangdeco}n_0 T_1 n_1 T_2\ldots n_{l-1} T_l n_l \qquad n_0=m, n_l=n\eq such that $T$ is equivalent to $T_1 T_2 \cdots T_l$.  
It is possible to express any $T$  (not uniquely) as a composition of elementary tangles.
\textit{Crossingless matchings} (section \ref{crossless}) are a special class of $(0,n)$ or $(n,0)$-tangles.  Given a set of $2n$ points on the plane, a crossingless matching is a set of $n$ nonintersecting curves each joining a pair of the given points. In the context of tangles a crossingless matching is regarded as a subset of $\bb{C}\times [0,1]$.
 \begin{definition}
 Let $\pz{C}_n $ be the set of isotopy (in $\mathbb{C}$) classes  of crossingless matchings between $2n$ points. \end{definition}The cardinality of $\pz{C}_n$ equals the $n$th Catalan number.

 \begin{figure}[ht]
\includegraphics[scale=0.4]{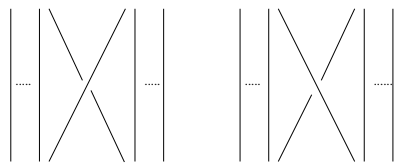}
\caption{The braids $\sigma_i$ and ${\sigma_i}^{t}$}
\label{elembraids}
\end{figure}

\begin{figure}[ht]
\includegraphics[width=4.5in]{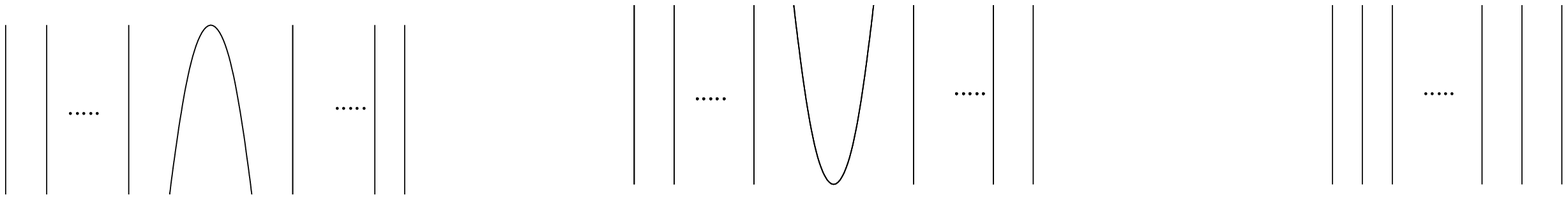}
\caption{$\cap_{i,m}$, $\cup_{i,m}$ and $id_m$}
\label{elementangles}
\end{figure}

  One can define a category $\mathbf{Tang}$ whose objects are natural numbers and $\hom(m,n)$ consists of equivalence classes of $(m,n)$-tangles. 
  $\mathbf{Tang}$ has a monoidal structure given by putting two tangles $kTl$ and $mTn$ ``side-by-side'' and obtain a $(k+m,l+n)$-tangle. We denote this by $T\oplus T'$. To each $(m,n)$-tangle $T$ there is assigned a ``mirror image'' $T^t$ which is a $(n,m)$-tangle.
   There is a generators and relations description of $\mathbf{Tang}$ due to Yetter \cite{Yetter} whose proof relies on Reidemeister's description of plane diagram moves.


\begin{lemma}[Yetter \cite{Yetter}]
The following  are all the commutation relations between elementary tangles where ``$=$'' means equivalence. If $|i-j|>1$ we have:\\
\begin{eqnarray}
\sigma_i\sigma_j&=&\sigma_j\sigma_i\\
\cap_{i;m}\cup_{j;m}&=&\cup_{j-2,m}\cap_{i;m}\\
\cap_{i;m}\sigma_{j;m}=\sigma_{j-2;}\cap_{i;m}  &\quad& \cup_{i;m}\sigma_{j;m}=\sigma_{j+2;m}\cup_{i;m}.
\end{eqnarray}

And for any $i$ we have:
\begin{eqnarray}
\sigma_{i}\cup_{i}&=&\cup_{i}\\
\sigma_i \sigma_i^t&=&id\\
\sigma_i\sigma_{i+1}\sigma_i&=&\sigma_{i+1}\sigma_i\sigma_{i+1}\\
\label{capcup}\cap_{i;m}\cup_{i+1;m}&=&id_{m-1} \\
\sigma_i \cup_{i+1} =\sigma_{i+1}^t\cup_i &\qquad& \sigma_i^t \cup_{i+1}=\sigma_{i+1}\cup_i.
\end{eqnarray}
\end{lemma}

\vspace{.6cm}

Given two decompositions of two equivalent tangles, the following lemma provides a natural  way of converting one to the other.
\begin{lemma}\label{tangledecomp}
If $T, T'$ are two equivalent tangles and $D:T=T_1T_2\cdots T_m$, $D':T'=T'_1T'_2\cdots T'_m$ are decompositions into elementary tangles then one decomposition can be converted to the other by a sequence of the above moves.
\end{lemma}
\begin{proof}
We can regard each decomposition as being inside $\mathbb{C}\times [0,1]$. We can also find Morse functions $f,f'$ on $T$ and $T'$ respectively such that
the decomposition of $T$ (resp $T'$) by critical sets of $f$ (resp. $f'$) yields $D$ (resp. $D'$); in other words $f$ has critical points separating each $T_i$ from $T_{i+1}$ and no more. There is a diffeomorphism $\phi$ of $\mathbb{C}\times [0,1]$ such that $\phi(T)=T'$.     The reason is that because $T$ and $T'$ are equivalent,    there is a family $T(t)$ of tangles with the same boundary points such that $T(0)=T,T(1)=T'$. 
Because strands of $T(t)$ never intersect, we can obtain a time dependent vector field on $T(t)$ by differentiation. This vector field can be smoothly extended to $\bb{C}\times [0,1]$. Let $\phi$ be the time one map of this vector field. Then $\phi(T)=T'$.

Let $f'' =\phi^*f'$ and $T''_i=\phi^{-1}(T'_i)$. So we get a decomposition $D'':T''_1,\cdots, T''_m$ of $T$ induced by $f''$.  According to Cerf theory \cite{Cerf}, there is a family $f_t$ of smooth functions such that $f_0=f, f_1=f''$ and $f_t$ is Morse except for finitely many times $t_1,\cdots, t_k$ and at each $t_i$ a pair of canceling critical points is introduced or deleted. Since $T$ is one dimensional, each $t_i$ has the effect of either (1) merging two adjacent handles (i.e. two adjacent elementary tangles) or decomposing one into two or (2) the effect of the  move \eqref{capcup} above.  In the first case the effect of merging two adjacent handles  and then separating into two new ones is equivalent to one of the  above moves.

Now between each $t_i,t_{i+1}$, each handle can be isotoped  to an equivalent one. The only way this can change the decomposition is by changing the value of $f_t$ at critical points and thereby changing the order of the critical points in the decomposition.  This also has the effect of commuting the handlebodies in the decomposition.

\end{proof}

The invariant that we define in this paper is an invariant of \textit{oriented tangles}. An oriented tangle comes with an orientation of each one of its components.  Two example are shown in Figure \ref{posnegbraid}. When considering commutation relations between tangles we ignore the orientation. 

\begin{figure}[ht]
\includegraphics[scale=0.6]{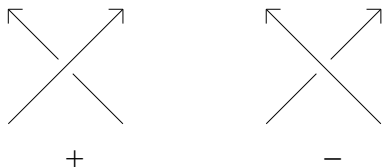}
\caption{Oriented braids $\sig^+$ and $\sig^-$}
\label{posnegbraid}
\end{figure}

\subsection{The functor}\label{func}

Let \bq\label{tdecomp}n_0 T_1 n_1 T_2\ldots n_{l-1} T_l n_l\eq be a decomposition of an oriented tangle $T$.
%
Set $\nu_j=i(T_j)$ and $\nu_{l+1}=t(T_l)$
 We have $\nu_i\in \mathrm{Conf}_{n_i}$ for $i=0,\ldots, l$.
To each $T_i$ we want to associate a Lagrangian correspondence $L_{i,i+1}=L_{T_i}$ between $\mathcal{Y}_{n_i,\nu_i}$ and $\mathcal{Y}_{n_{i+1},\nu_{i+1}}$. 
In this way we can associate to $T $ a generalized Lagrangian correspondence\\
 \begin{equation}\label{Phi}
 \Phi(T)=(L_{0,1},L_{1,2},\ldots, L_{n-1,n})\{ -m-w\}
 \end{equation}
\\
from $\mathcal{Y}_{n}$ to $\mathcal{Y}_{m}$. Here $m$ and $w$ are the number of cups and the writhe (number of positive crossings minus the number of negative ones) of the decomposition respectively. 

If $T_k$ is an elementary  braid in $Br_{2m}$, we set $L_{T_k}$ to be $ \on{graph} (h^{res}_\beta)$   
regardless of the orientation of the braid. Of course we can extend this definition to any braid.
We note that the symplectomorphisms $h^{res}_\beta$ are defined only on compact submanifols of the $\cal{Y}_m$ however this does not cause any problem since we are going to take the $\mathfrak{C}_c$ of these Lagrangians.
%
Let $V_i$ be the relative vanishing cycle  for the map $f$ in Lemma \ref{2nbhd} where $i$th and $(i+1)$th eigenvalues ($\mu_i, \mu_{i+1}$) of $\nu_k$ come together  at some point $\mu$. 
 There is a theorem of T.~Perutz (generalizing a former result of P.~Seidel) which describes monodromy maps around singularities of symplectic Morse-Bott fibrations as Dehn twists. Recall that  a symplectic Morse-Bott fibration (also called Lefschetz-Bott fibrations) over a disc $D$ consists of an almost complex  manifold $(E,J)$, a closed two-form $\Omega$ on $E$ and a $J$-holomorphic map $\pi: E\ra D$ such that the critical point set of $\pi$ is a submanifold of $E$  and the complex Hessian matrix of $\pi$ is nondegenerate. The form $\Omega$ is required to be a symplectic form when restricted to each fiber.

\begin{theorem}[Perutz \cite{PerutzI}, Theorem 2.19]
 Let $\pi: E \rightarrow D$ be  a symplectic Morse-Bott fibration  over the closed disc $D$ which has only the origin as singular value. If in addition the fibration is normally Kahler, ie.  a neighborhood of the critical point set of $\pi$ 
is foliated by $J$-complex normal slices  such that $J$  restricted to each slice  is integrable and the restriction of $\Omega$ to each fiber is Kahler, then
 the monodromy map around  the origin is Hamiltonian isotopic to the fibred Dehn twist $\tau_V$ along the vanishing cycle $V$ for the map $\pi$.
\end{theorem}
Therefore using the local picture of the Lemma \ref{2nbhd} we see that if we have a subset  $B\subset \cal{Y}_m$ for which the naive (non-rescaled) parallel transport map $h_{\sig_i}|_B$ is well-defined then \bq\label{monodromydehn}h_{\sigma_i}\cong\tau_{V_i}.\eq
The reason is that since the naive parallel transport is well-defined for all points of $B$, we can shrink the rescaling parameter in \eqref{tranportvec} to zero and thereby isotope $h_{\sig_i}^{res}$ to $h_{\sig_i}$.

\begin{lemma}\label{norotate}
Let $\nu=\{\mu_1,\ldots,\mu_{2m}\}\in \mathrm{Conf}_{2m}^0$ and  $\gamma: [0,1]\rightarrow \frak{h}/W$ such that $\gamma(0)=\nu$ and as $s\rightarrow 1$, $\mu_{2i-1}$ and $\mu_{2i}$ approach each other linearly and collide at some $\mu$. 
Then $h^{res}_{\sigma_i}$ restricted to $\Sm^{sub,\mu}$ is identity. Similarly if $\mu_{2k-1},\mu_{2k},\mu_{2k+1}$ come together at some $\mu$ then $\sig_{i}$ and $\sig_{i+1}$ act trivially on $\Sm^{sub3,\mu}$.
\end{lemma}

\begin{proof}
 We use the picture of Lemma \ref{3nbhd}. 
Both statements follow from \eqref{monodromydehn} and the fact that fibred  Dehn twists are identity outside  a small \nbhd of the spherical fibres.
\end{proof}

Let ${V_i}_x$ denote the ($S^2$) fiber of $V_i$ over $x$. We grade $\tau_{V_i}$ in such a way that \bq\label{gradedehn}\tau_{V_i} {V_i}_x={V_i}_x\{1\}\eq and the grading function vanishes outside a \nbhd of $V_i$. This grading is unique. (Lemma 5.6 in \cite{gradedlag})

\begin{remark}\label{khsremark}
 Monodromy actions of braid group on symplectic manifolds were first constructed in \cite{KhS}. Parallel transport maps $h^{res}_{\beta}$  form a homomorphism  from $\pi_1(\mathrm{Conf}_{2m})=Br_{2m}$ into the $\pi_0$ of $Sympl(\cal{Y}_m)$. In particular symplectomorphisms associated to elementary braids satisfy Artin's commutation relations. (Symplectic manifolds used in \cite{KhS} are compact with boundary. The manifolds $\cal{Y}_m$ in \cite{SS} that we use here can be obtained from them by attaching an infinite cone.)
\end{remark}




If $T_i=\cup_{j;m}$, we define a  Lagrangian   $L_{\cup_{k;m}}$, regardless of the orientation of $\cup_{k;m}$,  as follows. 
The result depends on a real parameter $R>0$.
To simplify the notation we set $k=2j-1, l=2j$.
%
 
With $\nu_i$ as given above let $\nu=\nu_i=\{\mu_1,\ldots, \mu_{2m}\}$.
 Let $\gamma$ be a curve in $\mathrm{Conf}^0_{2m}$ such that $\gamma(0)=\nu_i$ and as $s\rightarrow 1$, $\mu_{k}$ and $\mu_{l}$ approach each other linearly and collide at a point $\mu'$.  For example we can take $$\gamma(t)=\{\mu_1,\ldots,\mu_k+t(\mu_l-\mu_k)/2,\ldots,\mu_l-t (\mu_l-\mu_k)/2 ,\ldots,\mu_{2m} \}$$  provided that $\mu_k+t(\mu_l-\mu_k)/2$ does not intersect the other $\mu_i$. Set $\nu^{k,l}=\nu \backslash \{\mu_{k},\mu_{l}  \}$, $\nu'=\gamma(1)$.  
  We use Lemma \ref{2nbhd} to identify a \nbhd of $\Sm^{sub,\mu'}$ in $\Sm$ locally with $\Sm^{sub,\mu'} \times \mathbb{C}^3$. This induces a Kahler form and hence a metric on $\Sm^{sub,\mu'}\times \mathbb{C}^3$. We perturb the complex structure outside a compact ball of radius $\rho$ (to be chosen below) so that outside that set the resulting metric equals the product  metric. Now we use the relative vanishing cycle construction for the whole $\Sm^{sub,\mu'}$.  It  yields (after restriction) a sphere bundle $V=V_{\gamma(1-s)}(\Sm^{sub,\mu'}) \subset \mathcal{Y}_{m,\gamma(1-s)}$ for small $s$  with projection $\pi:V\rightarrow \mathcal{Y}_{m,\nu'}\cap O^{sub,\mu'_k}$. This construction works because the metric equals the product metric outside a compact set.

We denote the image of $V$ under  parallel transport map along $-\gamma$, i.e.  $  h^{-1}_{\gamma|_{[0,1-s]}}(V)\subset\mathcal{Y}_{m,\nu}$, by the same notation $V$.
 Composing $\pi$ with the  parallel transport map $h^{-1}_{\gamma|_{[0,1-s]}}$ we obtain a projection $\pi:V\rightarrow  \mathcal{Y}_{m,\nu'}\cap O^{sub,\mu'}$  which is a $S^2$ bundle. By Lemma \ref{2nbhd}, $ \mathcal{Y}_{m,\nu'}\cap O^{sub,\mu'}$ can be identified with $\cal{D}_{m-1,\nu'}$  from (\ref{Ddef}). Let $\delta$ be a geodesic in $\mathrm{Conf}^0_{2m}$  
 joining $\nu'$ to $\nu^{k,l}$.
 We can use parallel transport map \eqref{paralleltransD} along the curve $\delta$ to map  $\cal{D}_{m-1,\nu'}$ to $\cal{D}_{m-1,\nu^{k,l}\cup \{0,0\}}$. The latter can be identified with $\mathcal{Y}_{m-1,\nu^{k,l}}$. So we obtain a fibration $\pi:V\rightarrow \mathcal{Y}_{m-1,\nu^{k,l}}$. We can use this map $\pi$ to pull $V$ back to $\mathcal{Y}_{m-1,\nu^{k,l}}\times \mathcal{Y}_{m-1,\nu^{k,l}} $. Let ${\cup_{j;m}}$ be its restriction to the diagonal. It  is a Lagrangian submanifold of $\mathcal{Y}_{m,\nu_i}^- \times \mathcal{Y}_{m-1,\nu_i^{k,l}}=\mathcal{Y}_{m,\nu_i}^-\times\cal{Y}_{m-1,\nu_{i+1}}$.
 Let $\psi=\psi_1+\psi_2$ be the plurisubharmonic function on $\mathcal{Y}_{m,\nu_i}^-\times\cal{Y}_{m-1,\nu_{i+1}}$. We can choose $\rho$ in such a way that the inverse image of $\psi=R$ lies inside the ball of radius $\rho$.
We have a projection $\pi: L_{\cup_{j;m}}\rightarrow \Delta\subset \mathcal{Y}_{m-1,\nu_{i+1}}^-\times\cal{Y}_{m-1,\nu_{i+1}}$.\\

  As in the case of Lagrangians from crossingless matchings, 
   replacing the curve $\gamma$  with another curve in the same homotopy class (with fixed endpoints) results in a new $L{\cup_{j;m}}$ which is Lagrangian isotopic to the former one. Sine the first homology group of this Lagrangian is zero, this isotopy is exact. 

    We grade $L_{\cup_{i;m}}$ as follows. Lemma \ref{lag_commut} below tells us that fibers of $L_{\cup_i}$ and $L_{\cup_{i+1}}$ over each point of the diagonal intersect transversely at only one point.  We choose the grading in such a way that the absolute Maslov index of this intersection point (with regard to the two $S^2$ fibres) equals one.
    We can use Lemma \ref{2nbhd} and isotope $L_{\cup_i}$ to $\Delta \times \sqrt{z} S^2$ for some small $z\in \mathbb{C}$. One can see by direct computation that $\alpha_{L_C}=\alpha_{\Delta}\cdot \frac{z}{|z|}$.  The function $\alpha_{\Delta}$ is  constant. This means that
    the grading on $L_{\cap_i}$ is determined by the choice of a branch of $\arg(z)$. In particular such a grading is a constant function:
    \bq\label{gradingforcup} \tilde{L}_{\cup_i}=c_i. \eq
    We choose this $c_i$ to be the same for every $i$. This together with the formula \eqref{degdef2} (for $n=2$) imply that the absolute Maslov index of each fiberwise intersection point in $L_{\cup_i}\cap L_{\cup_{i+1}}$ equals $1$.
  Construction for $\cap_{k;m}$ is similar.\\


  The following lemma insures  that the Lagrangian we assign to crossingless matching agrees with that of Seidel and Smith.

\begin{lemma}\label{crosslesstocup}
If $C\in \pz{C}_m$ is a crossingless matching and arcs of $C$ are isotopic to $T_{i_1,j_1},\ldots, T_{i_m,j_m}$ where each $T$ is a cap or a cup then $L_C$ is isotopic to
$L_{T_{i_1,j_1}}\circ \cdots\circ L_{T_{i_m,j_m}}$.
\end{lemma}
\begin{proof}
We use induction on $m$. The base case is vacuous. For the induction step we note that our construction is the same as the induction step in the construction of Seidel and Smith (section \ref{crossless} above) except for the base of the fibration.
\end{proof}

In order for $\Phi$ to define a functor, we must verify that the above correspondences satisfy the same commutation relations as the tangles they are associated to. 
 First we have the following 
 cf. remark \ref{khsremark}.

\begin{lemma}\label{seidelth}
We have $L_{\sigma_i}L_{\sigma_{i+1}}L_{\sigma_i}=L_{\sigma_{i+1}}L_{\sigma_i}L_{\sigma_{i+1}}$ and if $|i-j|>1$, $L_{\sigma_i}L_{\sigma_j}=L_{\sigma_j}L_{\sigma_i}.$\\
\end{lemma}



\begin{lemma}\label{rotate_lag}
We have
\bq L_{\cap_{i;m}}L_{\sigma_{j}}\simeq L_{\sigma_{j-2}}L_{\cap_{i;m}}  \qquad L_{\cup_{i;m}}L_{\sigma_{j}}\simeq L_{\sigma_{j+2}}L_{\cup_{i;m}} \eq
if $|i-j|>1$ and for any $i$ we have:
 \begin{eqnarray}
\label{2tolast}L_{\cap_{i;m}}L_{\sigma_{i}}\simeq L_{\cap_{i;m}}\{ 1\}   &\qquad& L_{\sigma_{i}}L_{\cup_{i;m}}\simeq L_{\cup_{i;m}}\{ 1\} \\
\label{last}L_{\sigma_{i\phantom{;m}}} L_{\cup_{i+1}}\simeq L_{\sigma_{i+1}}^t L_{\cup_i }  &\qquad& L_{\sigma_{i+1}}L_{ \cup_{i}}\simeq L_{\sigma_{i}}^{t}L_{\cup_{i+1}}
\end{eqnarray}
where ``$\simeq$'' means exact isotopy.
\end{lemma}

\begin{proof}

Let $\beta:[0,1]\rightarrow \mathrm{Conf}^0_{2m}$ be a braid such that $\beta(0)=\beta(1)=\nu$. Note that in the construction of $L_\cap$   if we replace the curve $\gamma$ with
$\beta*\gamma$ and also replace $\delta $  with $\delta*\alpha$ where $\alpha$ joins $\nu^{(k)}$ to $\nu\backslash \{\beta(2k-1),\beta(2k) \}$, the construction will yield $h_\beta (\cap_{k;m})$. In general the basepoint $\nu^{(k)}$ or $\nu\backslash \{\beta(2k-1),\beta(2k) \}$ is of no importance so we can assume $\alpha$ to be constant.
If $\beta$ equals $\sigma_{k}$  then $\beta*\gamma$ joins $\mu_{2k-1}, \mu_{2k}$  and fixes the other eigenvalues so the construction will yield the same $L_{\cap_{k}}$. 
 If  $\beta=\sigma_{k+1}$ then $\beta*\gamma$ the result will be the same as starting with $\cap_{i+1}$ and using $\beta=\sigma_k$.  These two facts imply the above isotopies ignoring grading. As another proof which  shows equality of graded Lagrangians,
we appeal to \eqref{gradedehn} which implies \eqref{2tolast}. To prove \eqref{last}, we use Lemma 5.8 in \cite{gradedlag} which asserts if $L_0, L_1$ are two graded Lagrangian spheres whose intersection consists of only one point of Maslov index one and with the grading of the Dehn twists and chosen as above, we have
\bq\label{dehntwistequ}  \tau_{L_0}L_1 =\tau_{L_1}^{-1} L_0\eq as graded Lagrangians. Now we use Lemma \ref{3nbhd} to translate the picture into that of section \ref{fibered}. So we can identify $L_{\cup_{i}}$ with $\Lambda_{\alpha_1}\times_{S^1}\Delta$ and $L_{\cup_{i+1}}$ with $\Lambda_{\alpha_2}\times_{S^1}\Delta$. Because $h_{\sigma_i}$ and $h_{\sigma_{i+1}}$ act trivially on $\Delta$ by \ref{norotate}, we can identify them with Dehn twists around $\Lambda_{\alpha_1}$ and $\Lambda_{\alpha_2}$ respectively. Our choice of grading \eqref{gradingforcup} for $L_{\cup_i}$ implies that the hypothesis of  \eqref{dehntwistequ} are met. This immediately implies \eqref{last}. Note that because $\Lambda_{\alpha_i}$  are two dimensional, changing the order of them does not change the Maslov index of the intersection point.
\end{proof}



\begin{lemma}\label{lag_commut}
We have the following commutation relations where  ``$\simeq$'' means exact isotopy.

\begin{eqnarray}
\label{thesecondtolast}L_{\cap_{i,m}}\hspace{1ex} L_{\cup_{j,m\phantom{+1}}}&\simeq& L_{\cup_{j-2,m-2}}\hspace{1ex} L_{\cap_{i,m-2}} \quad \text{ if}  \quad|i-j|>1.\\
\label{thelast}L_{\cap_{i,m}} \hspace{1ex} L_{\cup_{i+1,m}} &\simeq& L_{id_{m-1}} \{-1\} \qquad \text{ for any } i.
\end{eqnarray}
\end{lemma}

\begin{proof}

To prove \eqref{thesecondtolast} let $\nu=\{\mu_1,...,\mu_{2m}\}$ and $\nu', \nu'', \nu'''$  be $\nu$ minus $\{\mu_{2i-1},\mu_{2i}\}$, $\{\mu_{2j-1},\mu_{2j}\}$,
   $\{\mu_{2i-1},\mu_{2i}, \mu_{2j-1},\mu_{2j}\}$ respectively. Therefore
\begin{eqnarray*}
L_{\cap_{i,k}}\subset\cal{Y}_{m-1,\nu'}\times \cal{Y}_{m,\nu\phantom{-2'''}} &\text{and}& L_{\cup_{j,m\phantom{-2}}}\subset \cal{Y}_{m,\nu}\times \cal{Y}_{m-1,\nu''}.\\
L_{\cup_{j-2,k-2}}\subset\cal{Y}_{m-1,\nu'}\times \cal{Y}_{m-2,\nu'''} &\text{and}& L_{\cup_{i,m-2}}\subset \cal{Y}_{m-2,\nu'''}\times \cal{Y}_{m-1,\nu''}.\\
\end{eqnarray*}
We have projections $\pi_i:L_{\cap_{i,m}}\rightarrow  \cal{Y}_{m-1,\nu'}$, $\pi_j: L_{\cup_{j,m}}\rightarrow  \cal{Y}_{m-1,\nu''}$ given in the construction of the these Lagrangians. We have\\
\begin{eqnarray*}
L_{\cap_{i,m\phantom{-2}}}L_{\cup_{j,m\phantom{-2}}}&=&\{(m_1,m_2)| \exists m , (m_1,m)\in L_{\cap_{i,k}}, (m,m_2)\in L_{\cup_{j,k}}  \}=\\
&=&\{ (m_1,m_2)| \exists m, m_1=\pi_i(m), m_2=\pi_j(m) \}.\\
L_{\cup_{j-2,m-2}} L_{\cap_{i,m-2}}&=&\{(m_1,m_2)| \exists m, (m_1,m)\in L_{\cup_{j-2,k-2}} , (m,m_2)\in L_{\cap_{i,k-2}} \}=\\
&=&\{(m_1,m_2)| \exists m, \pi_j(m_1)=m, \pi_i(m_2)=m \}\\&=&\{  (m_1,m_2)| \pi_j(m_1)=\pi_i(m_2) \}
\end{eqnarray*}
By the remark after the Lemma \ref{2nbhd}, these projections are given by projection to the first and second factor in $\mathfrak{sl}(E_{\mu_1})\oplus \mathfrak{sl}(E_{\mu_2})\oplus \zeta$   
  and so $\pi_i\pi_j=\pi_j\pi_i$ from which the equality of the compositions follows.\\






As for \eqref{thelast} we must show that $L_{\cap_i} \circ L_{\cup_{i+1}}$ equals the diagonal $\Delta \subset \mathcal{Y}_{m}\times \mathcal{Y}_{m}$.
Using Lemma \ref{rotate_lag} we see that 
 $$h_{\sigma_{i}\sigma_{i+1}} L_{\cap_i}=L_{\cap_{i+1}}.$$ 
So we are reduced to showing that  $L_{\cap_i}\cap h_{\sig_{i}\sig{i+1}}(L_{\cap_i})\simeq \Delta$.
Let $\nu=(\mu_1,...,\mu_{2m+2})$  be  such that  $\mu_i+\mu_{i+1}+\mu_{i+2}=0$ and\\ $\bar{\nu}=(\mu_1,...\mu_{i-1},0,\mu_{i+3},...,\mu_{2m+2})$ also 
$\hat{\nu}=(\mu_1,...,\mu_{i-1},0,0,0,\mu_{i+3},...,\mu_{2m+2})$.
We take $\mathcal{Y}_{m+1}=\mathcal{Y}_{m+1, \nu}$ and $\mathcal{Y}_{m}=\cal{Y}_{m,\bar{\nu}}$. The latter can be identified  with $\cal{D}_{m+1,\hat{\nu}}$ which is  $(\chi^{-1}(\hat{\nu})\cap \cal{S}_{m+1})^{sub3,0}$.
  %
Lemma \ref{3nbhd} gives us a $\Bbb{C}^4$-bundle $\cal{L}$ over $\cal{S}_{m+1}^{sub3,0}$ 
with a local symplectomorphism $\phi$  from $\cal{L}$ to a neighborhood  $U$ of $\cal{S}^{sub3,0}$ in $\cal{S}_{m+1}$. 
We  identify $\cal{Y}_{m+1,\nu}$ 
 with its image under $\phi$ and so $\Delta\subset \cal{L}\times \cal{S}_{m+1}^{sub3,0}$.  $\cal{Y}_{m,\bar{\nu}}$ lies in the zero section and $L_{\cap_i}$ is the pullback of a  sphere bundle over the zero section to $\Delta$.
Lemma  \ref{norotate} tells us that $h_{\sigma_{i}\sig_{i+1}}$ is identity when restricted to $\Delta$ 
 therefore   $h_{\sig_{i,i+1}}(L_{\cap_i})$ is another sphere bundle over $\Delta$. We show that fibers of this two bundles over each point of $\Delta$ intersect at only one point so $L_{\cap_i} \cap h_{\sigma_{i}\sig_{i+1}}(L_{\cap_i})$ is Lagrangian isotopic  to $\Delta$. 
Since these Lagrangians have vanishing first cohomology groups   this isotopy is exact.\\

Let $L_1$and $L_2$  be fibers of $L_{\cap_i}$ and $h_{\sig_{i}\sig_{i+1}}(L_{\cap_i})$ over a point of $\Delta$. So \bq \label{lsigma}L_{2}=h_{\sig_{i}\sig_{i+1}}(L_1).\eq
By Lemma \ref{2lagrang}, $L_{\cap_i}$ is isotopic to $\Lambda_{d,\alpha_1}=\Delta  \times_{S^1} \Lambda_{\alpha}$ so  $L_1$ can be identified with $\Lambda_{\alpha_1}$. When we perform $h_{\sigma_{i}\sig_{i+2}}$, the leftmost and the rightmost zeros in Figure \ref{curvesalphabeta} (which are the ``fiberwise'' eigenvalues) get swapped; therefore $\mathbf{c}$ is sent to $\mathbf{c}'$ and so $L_2=\Lambda_{\mathbf{c}}$.
Since $\Lambda_{\mathbf{c}}=\bigcup_{r=0}^1   C_{d,z,\mathbf{c}(r)}$  and $ C_{d,z,a}$ is given by  $C_{d,z,a}=\{(b,c):  |b|=|c|, a^3-da-z=-bc \} $, 
  we see that since the curves $\mathbf{c} $ and $\mathbf{c}''$ intersect at only one point (i.e. the first root of $a^3-ad-z=0$) then $L_1=\Lambda_{\alpha_1}$ and $L_2=\Lambda_{\alpha_2}$ intersect at only one point so we are done. \\

We chose  the gradings  \eqref{gradingforcup} for $L_{\cap_i}$ to be a constant function $c$ and be the same for all $i$.
Let $S_{ix}$ be the fiber of $L_{\cap_i}$ over $x$. Then the fiber of $L_{\cap_{i;m}}\circ  L_{\cup_{i+1;m}}$ over $x$ equals $S_{ix}\circ S_{i+1x}^t$. Grading of $S_{i+1x}^t$ equal $1-\tilde{S}_{i+1x}$ by \eqref{neggrade}. Therefore the grading of the point  $S_{ix}\circ S_{i+1x}^t$ is $c+1-c=1$ which implies \eqref{thelast}.

\end{proof}

\begin{theorem}\label{wellfunc}  The assignment $\Phi$  in \eqref{Phi} gives a functor from the category of tangles to the symplectic category so it defines a graded symplectic valued topological field theory.
\end{theorem}
\begin{proof}
Follows from \ref{tangledecomp} and \ref{seidelth} to \ref{lag_commut}. We see that difference in grading for each commutation relation gets canceled by the change in the writhe plus number of cups. The only commutation relations which involve grading shift are \eqref{2tolast} and \eqref{thelast} which happen to be the only ones involving change in $-m-w$.   For   \eqref{thelast}, $-m$ plus the degree shift is equal on both sides of the equation. Note that if $T$ contains $\cap_i(\sig_i^\epsilon)^\pm$  where $\epsilon=1,-1$  and $\pm$ is the writhe of $\sig _i$ then $\pm\epsilon$ has to be equal to $-1$.
This implies that $-w$ plus the grading shift is equal on both sides of \eqref{2tolast}.
\end{proof}

\subsection{ The invariant}

We can obtain tangle invariants at two levels from the symplectic valued topological field theory $\Phi$: a functor valued invariant and a group valued one.
 According to \cite{MWW} the \textit{generalized Fukaya category}   $\Fuk^\#(M)$ of a symplectic manifold $M$  is an $A_\infty$ category whose objects are compact generalized Lagrangian correspondences between a point and $M$ and morphisms between two such objects $\underline{L}_0$ and $\underline{L}_1$ are the elements of  Floer chain complex for  $\underline{L}_0^t \# \underline{L}_1$.
The $A_\infty$ structure on $Fuk^\#(M)$ is given by counting holomorphic ``quilted polygons''. More precisely the maps
\bq\label{mu_d}
\mu^d: CF(\underline{L}_0, \underline{L}_1)\otimes \cdots \otimes CF(\underline{L}_{d-1},\underline{L}_{d}) \rightarrow CF( \underline{L}_0, \underline{L}_d )   
\eq
are given by counting quilted $(d+1)$-gons whose boundary conditions are given by the $\underline{L}_i$.\\ 

 Here we need an enlargement of (generalized) Fukaya category of a Stein manifold to include  noncompact admissible Lagrangians.
 For a Stein manifold $M$ we denote by $\on{Fuk}^\#(M)$ the $A_\infty$ category whose objects are allowable proper generalized Lagrangian correspondences between $M$ and a point. Symplectic manifolds involved in these correspondences can be either Stein or compact.
Two such correspondences $\underline{L}_0$ and $\underline{L}_1$ satisfy the finite intersection property and we define the set of morphisms between them to be   $CF(\mathfrak{C}_c (\underline{L}_0^t \# \underline{L}_1) )$.
 Here 
we choose $c$ so as to  include the intersection points of the Lagrangians. 
 The $A_\infty$ structure is given in the same way as for the generalized Fukaya category of compact manifolds. Lemma \ref{compact} insures that the moduli spaces involved are compact.

  For an $(m,n)$  tangle $T$,  let $\Phi(T)=(L_{0,1},L_{1,2},...,L_{n-1,n})$ be as in \eqref{Phi}. We obtain an $A_\infty$ functor
\bq\label{funcvalued}\Phi_T^\#: \Fuk^\#(\cal{Y}_m)\rightarrow \Fuk^\#(\mathcal{Y}_n)\eq which is given at the level of objects by $\Phi_T^\#(\underline{L})=\underline{L}\#\Phi(T)$
for each $\underline{L}\in \Fuk^\#(\cal{Y}_m)$.  For details on the functor $^\#$ in general see \cite{MWW}. If $K$ is a link then 
we obtain a functor \bq\label{funcpoint}\Phi_K^\#:\Fuk^\#(pt)\rightarrow \Fuk^\#(pt). \eq\\
The group valued  tangle invariant is defined as follows.

\begin{definition}
 \bq\label{grpvalued}\mathpzc{Kh}_{symp}(T)=\bigoplus_{\substack{C\in\pz{C}_m\\C'\in \pz{C}_n}}  HF(L_C,\Phi(T),L_{C'}).\eq
\end{definition}

Each summand in the above direct sum is equal to the Floer cohomology of the Lagrangians
\begin{eqnarray*}
\mathcal{L}_0&=&L_C\times L_0\times L_{1,2}\times...L_{2k-1,2k}  \\
\mathcal{L}_1&=&L_{0,1}\times L_{2,3} \times ...\times L_{2k+1}\times L_{C'}
\end{eqnarray*}
 in $\cal{Y}=\cal{Y}^-_{n_0}\times\cal{Y}_{n_1}\times ...\times \cal{Y}_n $. If $\psi_i$ is the plurisubharmonic function on $\cal{Y}_i$ then $\cal{Y}$ is a Stein manifold with the plurisubharmonic function $\psi=\Sigma \psi_i$. 

\begin{lemma} The Lagrangians $\cal{L}_i$, $i=0,1$ are allowable. \end{lemma}
\begin{proof}
 The Lagrangians are exact since they are simply connected submanifolds of exact manifolds.
A point $m=(m_0,...,m_{2k+1})$ of tangency of $\mathcal{L}_0$ to a level set of $\psi$ is a critical point of $\psi|_{\mathcal{L}_0}$. Since $\psi$ is the sum of the plurisubharmonic functions on each $\cal{Y}_i$, $m_i$ is a critical point of $\pi_i^* \psi|_{L_{i,i+1}}$. If $L_{i,i+1}$ is noncompact, it is a vanishing cycle on the diagonal either in  $\cal{Y}_m\times\cal{Y}_{m+1}$ or in $\cal{Y}_{m+1}\times\cal{Y}_{m}$. In the first case $m_i$ is a critical point of $\pi_i^* \psi=\psi_i$ on $\cal{Y}_m$. Since $\cal{Y}_i$ and $\psi_i$ are algebraic, the critical point set is compact. The second case is similar.
\end{proof}

 \begin{theorem}\label{khswelldef}
 For any tangle $T$, $\khs(T)$ is well-defined and is independent of the decomposition of $T$ into elementary tangles.
  \end{theorem}
  \begin{proof}
 The finite intersection condition holds since the Lagrangian correspondences $L_{i,i+1}$ are proper (cf. \ref{propercorr}). We take the parameter $R$ in the construction of the  $L_{\cap_i}$ and $L_{\cup_i}$ so that all the intersection points are included in $M_{\leq R}$. Therefore by Proposition \ref{floerwell} the above Floer cohomology is well-defined.  Note that since our Lagrangians are products of  2-spheres, they have a unique spin structure so the Floer cohomology groups above are  modules over $\bb{Z}$. (cf. Theorem ``Fs'' in \cite{FOorientation}.)   
  Independence of the decomposition follows from the  Functoriality Theorem and theorem \ref{wellfunc}.
     \end{proof}

   It is clear that if $K$ is a $(0,0)$-tangle, i.e. a link, then the above invariant equals the original invariant of Seidel and Smith
  \eqref{defss}.


\subsection{Some computations}
We compute $\khs$ for elementary tangles. Set $$\pz{V}=H^*(S^2)\{-1\}.$$ Remember that $\oplus$ denotes unlinked disjoint union.
 Let $\sig^+_{i;m}$ and $\sig^-_{i;m}$ be as in the Figure \ref{posnegbraid}. 
%
%
If Lagrangians $L,L'\subset X\times \bb{C}^3$ are obtained from Lagrangians $K,K'\subset X$ by the relative vanishing cycle construction then we can isotope $L$ and $L'$ to $K\times S^2$ and $K'\times S^2$ inside a compact subset. Therefore we get \bq\label{Kuneth}  HF(L',L')=HF(K\times S^2,K'\times S^2)=HF(K,K')\otimes H^*(S^2)\eq
For a tangle $T$ it follows from the above formula that \bq\label{Kunethtang} \khs(T\oplus \bigcirc)=\khs(T)\otimes H^*(S^2)\{-1\}= \khs(T)\otimes \pz{V}.\eq The $-1$ degree shift here comes from the cup in $\bigcirc$. More generally it follows from the commutation relations in section \ref{func} that for any two tangles $kTl,mT'n$ we have
\bq\Phi^\#(T\oplus T')=\Phi^\#(T\oplus id_m)\circ \Phi^\#(id_l\oplus T')=\Phi^\#(id_k\oplus T')\circ  \Phi^\#(T\oplus id_n) .\eq

This gives an injection
\bq \khs(T\oplus T')\hookrightarrow \khs(T)\otimes \khs(T')\eq
For  links $L,L'$ this becomes an isomorphism    \bq\khs(L\oplus L')=\khs(L)\otimes\khs(L').\eq

Following Khovanov \cite{functorvalued} we define \bq\label{Hn}H^n=\khs(id_n)=\bigoplus_{\substack{C,C'\in\pz{C}_n}} HF(L_C^t,L_{C'}).\eq
It follows from \eqref{Kuneth} and \ref{crosslesstocup}  that each summand in the above direct sum equals $H^{*}(S^2)^{\otimes k}$ where $k$ is the number of circles made by attaching $C^t$ to $C'$. Therefore \eqref{Hn} 	equals Khovanov's $H^n$ as an abelian group.\\

We now give a recursive decomposition of $H^m$. Denote by $\P_m'$ the subset of $\P_m$ consisting of elements which contain $\cap_1$, i.e. elements which contain an arc between points 1 and 2,  and denote by $\P_m''$ its complement. (1 can be replaced with any $1\leq i\leq 2m-1$.) $\P_m'$ is in one-to-one correspondence with $\P_{m-1}$. We have a map $\P_m''\rightarrow \P_{m-1}$ $a\mapsto a'$ given by joining the two strands of $a$ that stem from 1 and 2. Let $\P\P_m^1\subset \P_m''\times \P_m''$ be the subset of all $(a,b)$ such that the arcs passing through points $1$ and $2$ in $a^tb$ form a single circle. If $(a,b)$ is in the complement of $\P\P_m^1$ then the arcs passing through points $1$ and $2$ in $a^tb$ form two circles. \\ 

Let  $a,b\in \P_m'$. If $\bar{a}$ denotes $a$ with the $\cap_1$ removed then we have $HF(L_a^t,L_b)=HF(\bar{a},\bar{b})\otimes \pz{V}\{1\}$
by \eqref{Kuneth}. This contributes a summand of $H^{m-1}\otimes\pz{V}\{1\} $ to $H^m$.
Set $${\tilde{H}^{m}}=\bigoplus_{a\in \P_m', b\in\P_m''} HF(L_a^t,L_b).$$  The embedded circle $C$ in $a^t b$  which passes through points 1 and 2 contributes a factor of $\pz{V}\{1+i\}$ to $HF(L_a^t,L_b)$ where $i$ is the number of other pairs of points $2k-1,2k$ which $C$ passes through. This follows from  \eqref{thelast}. We can set
 $$\tilde{H}^{m}={H^m}'\otimes \pz{V}\{1\}$$ 
where $\pz{V}\{1\}$ is the ``local'' contribution of the circle containing $\cap_1$ or $\cup_1$. This means that if $a\in \P_m', b\in\P_m''$ and we modify the strands of $a^t b$ passing through 1 and 2 only in a small \nbhd of the points 1 and 2 then we alter only the second factor in $ {H^m}'\otimes \pz{V}\{1\} $.
Also denote by $\tilde{H}^m_1$ and $\tilde{H}^m_2$ the contribution of $\P\P_m^1$ and its complement to $H^m$. Again we can write $\tilde{H}^m_1=H^m_1\otimes \pz{V}\{1\}$ and $\tilde{H}^m_2=H^m_2\otimes \pz{V}\{1\}\otimes \pz{V}\{1\}$ where $\pz{V}\{1\}$ resp. $\pz{V}\{1\}\otimes\pz{V}\{1\}$ are ``local''contributions from the single circle resp. the two circles formed by arcs passing through 1 and 2.
Therefore we get

\bq \label{HmHm-1} H^m=\Bigl(\left(H^{m-1}   \oplus {H^m}' \oplus {H^m}'\oplus H^m_1 \right)\otimes\pz{V}\{1\} \Bigr)\bigoplus H^m_2\otimes\pz{V}\{1\}\otimes \pz{V}\{1\}.\eq
as abelian groups.

\begin{lemma} We have
\begin{eqnarray*}
&\khs&(\sig_{i;m}^\pm)=\\
&& \Bigl(\left(H^{m-1}   \oplus {H^m}' \oplus {H^m}'\oplus H^m_1 \right)\otimes\pz{V}\{1\mp2\} \Bigr)\bigoplus H^m_2\otimes \pz{V}\{2\}.\\
&\khs&(\cup_{i;m})=(H^{m-1}\otimes \pz{V}) \bigoplus \oplus_{a\in \P_m'', b\in \P_{m-1}} HF(L_{a'},L_b) \\
&\khs&(\cap_{i;m})=(H^{m-1}\otimes \pz{V}\{1\})\bigoplus  \oplus_{a\in \P_m'', b\in \P_{m-1}} HF(L_{a'},L_b).
\end{eqnarray*}
\end{lemma}

\begin{proof}
The last two equalities are immediate consequences of \eqref{Kunethtang}. The computation for the first equality is similar to that of the rings $H^m$. The only difference comes from the existence of braids $\sig_{i;m}^\pm$. As stated before if we modify the strands of $a^t$ passing through 1 and 2 in a small \nbhd of these two points then only the factors  $\pz{V}\{1\}$  and $\pz{V}\{1\}\otimes \pz{V}\{1\}$ in \eqref{HmHm-1} change. In the first four direct summands of the decomposition \eqref{HmHm-1} the component of $a^t\sig_{i;m}^\pm b$ containing the braid contributes the following.\\

$HF(L_{\cap_i} L_{\sig_{i;m}^\pm},L_{\cup_i})\{-w(\sig_{i;m}^\pm)\}=HF(L_{\cap_i}\{\mp 1\},L_{\cup_i})\{\mp 1\}=\pz{V}\{1\mp 2\}$\\

The second to last equality is because $\sigma_{i;m}^+=\sig_{i;m}$ and $\sig_{i;m}^-=\sig_{i;m}^{-1}$ when we ignore the orientation.
In the last direct summand the component of $a^t\sig_{i;m}^\pm b$ containing the braid looks locally like the $T_0$ in Figure \ref{proofofkhsig}. (The figure depicts the case of $\sig_{i;m}^+$).  It is equivalent to $\cap_{1,2}\cap_{3,4}(\sig_{i;m}^\pm)\cup_{1,2} \cup_{3,4}$. We can use the commutation relation  \eqref{last} to get $\cap_{1,4}\cap_{2,3}{\sig_{i;m}^\pm}^{-1} \cup_{1,2}\cup_{3,4}$. Therefore the contribution of $T_0$ is
\begin{eqnarray*}
&HF&(L_{\cap_{1,4}}L_{\cap_{2,3}}L_{{\sig_{i;m}^\pm}^{(-1)}}, L_{\cup_{1,2}}L_{\cup_{3,4}})\{w(\sig_{i;m}^\pm)\}=\\
 &HF&(L_{\cap_{1,4}}L_{\cap_{2,3}}\{\pm 1\},L_{\cup_{1,2}}L_{\cup_{3,4}})\{w(\sig_{i;m}^\pm)\}=\\
&HF&(L_\cap,L_\cup)\{ 1\}\{\pm1 -w(\sig_{i;m}^\pm) \}=\pz{V}\{2\}.
\end{eqnarray*}
Where we have used commutation relations \eqref{2tolast} and\eqref{thelast}.

\end{proof}

\begin{figure}[ht]
\includegraphics[scale=0.8]{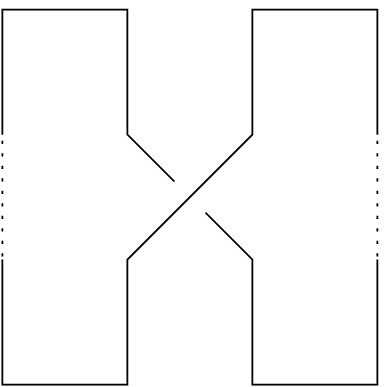}
\caption{   }
\label{proofofkhsig}
\end{figure}

\begin{remark}
These values agree with Khovanov invariant (after collapsing the bigrading) as abelian groups.
It turns out that   $H^n$ has a canonical ring structure for each $n$ and
the group $\mathpzc{Kh}_{sympl}(T)$, for a $(m,n)$-tangle $T$, is  a bimodule over the  $(H^m,H^n)$.
In a subsequent paper
  we  define a homomorphism  \bq\label{tensor}\mathpzc{Kh}_{sympl}(T')\otimes_{H^j}  \mathpzc{Kh}_{sympl}(T)\rightarrow \mathpzc{Kh}_{sympl}(T'\circ T)\eq for any two  tangles $iTj$ and $jT'k$.

\end{remark}

\bibliographystyle{plain} \bibliography{biblio}

\begin{thebibliography}{10}

\bibitem{CautisK1}
Sabin Cautis and Joel Kamnitzer.
\newblock Knot homology via derived categories of coherent sheaves. {I}. {T}he
  {${\mathfrak{sl}}(2)$}-case.
\newblock {\em Duke Math. J.}, 142(3):511--588, 2008.

\bibitem{Cerf}
Jean Cerf.
\newblock La stratification naturelle des espaces de fonctions
  diff\'erentiables r\'eelles et le th\'eor\`eme de la pseudo-isotopie.
\newblock {\em Inst. Hautes \'Etudes Sci. Publ. Math.}, (39):5--173, 1970.

\bibitem{EHS}
Y.~Eliashberg, H.~Hofer, and D.~Salamon.
\newblock Lagrangian intersections in contact geometry.
\newblock {\em Geom. Funct. Anal.}, 5(2):244--269, 1995.

\bibitem{FOorientation}
Kenji Fukaya, Yong-Geun Oh, Hiroshi Ohta, and Kaoru Ono.
\newblock {L}agrangian intersection homology-{A}nomaly and obstruction.
\newblock {\em preprint.}

\bibitem{Stablemanifoldthm}
Morris~W. Hirsch and Charles~C. Pugh.
\newblock Stable manifolds for hyperbolic sets.
\newblock {\em Bull. Amer. Math. Soc.}, 75:149--152, 1969.

\bibitem{humphreysbook}
James~E. Humphreys.
\newblock {\em Introduction to {L}ie algebras and representation theory},
  volume~9 of {\em Graduate Texts in Mathematics}.
\newblock Springer-Verlag, New York, 1978.
\newblock Second printing, revised.

\bibitem{JMorozov}
Nathan Jacobson.
\newblock Completely reducible {L}ie algebras of linear transformations.
\newblock {\em Proc. Amer. Math. Soc.}, 2:105--113, 1951.

\bibitem{Kamnitzer}
Joel Kamnitzer.
\newblock The beilinson-drinfeld grassmannian and symplectic knot homology.
\newblock {\em arXiv:0811.1730}.

\bibitem{categorification}
Mikhail Khovanov.
\newblock A categorification of the {J}ones polynomial.
\newblock {\em Duke Math. J.}, 101(3):359--426, 2000.

\bibitem{functorvalued}
Mikhail Khovanov.
\newblock A functor-valued invariant of tangles.
\newblock {\em Algebr. Geom. Topol.}, 2:665--741 (electronic), 2002.

\bibitem{KhS}
Mikhail Khovanov and Paul Seidel.
\newblock Quivers, {F}loer cohomology, and braid group actions.
\newblock {\em J. Amer. Math. Soc.}, 15(1):203--271 (electronic), 2002.

\bibitem{ChiprianSS}
Ciprian Manolescu.
\newblock Nilpotent slices, {H}ilbert schemes, and the {J}ones polynomial.
\newblock {\em Duke Math. J.}, 132(2):311--369, 2006.

\bibitem{MWW}
S.~Ma'u, K.~Wehrheim, and Woodward C.
\newblock ${A}_\infty$ functors for {L}agrangian correpondences.
\newblock {\em to appear}.

\bibitem{OHnoncompact}
Yong-Geun Oh.
\newblock Floer homology and its continuity for non-compact {L}agrangian
  submanifolds.
\newblock {\em Turkish J. Math.}, 25(1):103--124, 2001.

\bibitem{PerutzI}
Tim Perutz.
\newblock Lagrangian matching invariants for fibred four-manifolds. {I}.
\newblock {\em Geom. Topol.}, 11:759--828, 2007.

\bibitem{gradedlag}
Paul Seidel.
\newblock Graded {L}agrangian submanifolds.
\newblock {\em Bull. Soc. Math. France}, 128(1):103--149, 2000.

\bibitem{SS}
Paul Seidel and Ivan Smith.
\newblock A link invariant from the symplectic geometry of nilpotent slices.
\newblock {\em Duke Math. J.}, 134(3):453--514, 2006.

\bibitem{Slodowy}
Peter Slodowy.
\newblock {\em Simple singularities and simple algebraic groups}, volume 815 of
  {\em Lecture Notes in Mathematics}.
\newblock Springer, Berlin, 1980.

\bibitem{WW}
Katrin Wehrheim and Christopher~T. Woodward.
\newblock Functoriality for {L}agrangian correspondences.
\newblock {\em arXiv:0708.2851v1}.

\bibitem{weisteincat}
Alan Weinstein.
\newblock The symplectic ``category''.
\newblock In {\em Differential geometric methods in mathematical physics
  (Clausthal, 1980)}, volume 905 of {\em Lecture Notes in Math.}, pages 45--51.
  Springer, Berlin, 1982.

\bibitem{Yetter}
David~N. Yetter.
\newblock Markov algebras.
\newblock In {\em Braids (Santa Cruz, CA, 1986)}, volume~78 of {\em Contemp.
  Math.}, pages 705--730. Amer. Math. Soc., Providence, RI, 1988.

\end{thebibliography}

\verb"Department of Mathematics, Rutgers University, New Brunswick, NJ"

 \verb"rezarez@math.rutgers.edu"

\end{document}